\documentclass{article}
\usepackage[utf8]{inputenc}
\usepackage[T1]{fontenc}
\usepackage[english]{babel}
\usepackage{enumerate}
\usepackage{graphicx} 
\usepackage{subcaption}								
\usepackage{booktabs}
\usepackage{mathtools}
\usepackage{amsmath,bm}
\usepackage{amssymb}
\usepackage{amsfonts} 
\usepackage{fancyhdr}
\usepackage{fourier} 
\usepackage{pgfplots} 
\usepackage{cancel}
\usepackage{amsthm}
\usepackage{epstopdf}
\ifpdf
\DeclareGraphicsExtensions{.eps,.pdf,.png,.jpg}
\else
\DeclareGraphicsExtensions{.eps}
\fi
 
\usepackage[square,sort,comma,numbers]{natbib}                               
\usepackage{tikz-cd}

\usepackage{listings}								% Enables source code listings
\usepackage[top=3cm, bottom=3cm,
			inner=3cm, outer=3cm]{geometry}			% Page margin lengths	 		
\usepackage{eso-pic}								% Create cover page background

\usepackage{float} 									% Enables object position enforcement using [H]
\usepackage{parskip}								% Enables vertical spaces correctly 
\usepackage{tcolorbox}

\newcommand{\RR}{\mathbb{R}}

\newcommand{\RT}{\mathbf{RT}}

\newcommand{\permi}{\eta}
\DeclareMathOperator{\dive}{div}

\newcommand{\jump}[1]{[ #1 ]}
\newcommand{\bfV}{\mathbf{V}}

\newcommand{\bfK}{\mathbf{K}}
\newcommand{\bfH}{\mathbf{H}}
\newcommand{\bfE}{\mathbf{E}}

\newcommand{\bfx}{\mathbf{x}}

\newcommand{\bfu}{\pmb{u}}
\newcommand{\bfv}{\pmb{v}}
\newcommand{\bfn}{\pmb{n}}
\newcommand{\bft}{\pmb{t}}
\newcommand{\bfrho}{\pmb{\rho}}

\newcommand{\bfw}{\pmb{w}}
\newcommand{\boundu}{\partial\Omega_{\pmb{u}}}
\newcommand{\Es}{E^{\Sigma}}
\newcommand{\mesh}{\mathcal{T}_h}
\newcommand{\mcT}{\mathcal{T}}
\newcommand{\mcM}{\mathcal{M}}

\newcommand{\mcF}{\mathcal{F}}

\newcommand{\uB}{u_B}

\newcommand{\kp}{\hat{k}}%{k-1}

\newcommand{\mcTT}{\mcT_{h}(T_M)} % macro element patch to T
\newcommand{\mcTL}{\mcT_{H}} %{\mcT_{H,i}}
\newcommand{\mcFM}{\mcF(\mcM_h)} % interior faces to macro element M
\newcommand{\mcFs}{\mcF_{\Sigma}} % faces
\newcommand{\mcFsi}{\mcF^{\circ}_{\Sigma}} % interior faces
\newcommand{\tn}{|\mspace{-1mu}|\mspace{-1mu}|}

\newcommand{\Vk}{\bfV_{h,k}}
\newcommand{\Pk}{P_{h,k}}
\newcommand{\Qk}{Q_{h,k}}
\newcommand{\Qks}{Q_{h,k+1}^{\Sigma}}
\newcommand{\VkZ}{\bfV_{h,0}}

\newcommand{\QkZ}{Q_{h,0}}
\newcommand{\QksZ}{Q_{h,1}^{\Sigma}}

\newcommand{\Af}{\mathcal{A}}
\newcommand{\Mf}{\mathcal{M}}
\newcommand{\Bf}{\mathcal{B}}

\newcommand{\Omeps}{\Omega_{\epsilon}} 
\newcommand{\Omsig}{\Omega_{\Sigma}}  
 
\newcommand{\Omdh}{\Omega_{\mcT}} 
\newcommand{\interpv}{\pi_{h}}
\newcommand{\projp}{\Pi_{h}}
\newcommand{\projps}{\Pi_{h}^{\Sigma}}

\newcommand{\Hdivr}{\bfH^{1,\dive}}

\theoremstyle{plain}
\newtheorem{thm}{Theorem}[section]
\newtheorem{remark}[thm]{Remark}

\newtheorem{lemma}[thm]{Lemma}
 
\numberwithin{equation}{section}					% Numbering order for equations
\numberwithin{figure}{section}	 					% Numbering order for figures
\numberwithin{table}{section}						% Numbering order for tables

\title{Stabilized Lagrange Multipliers for Dirichlet Boundary Conditions in Divergence Preserving Unfitted Methods}

\author{Thomas Frachon \footnote{Department of Mathematics, KTH Royal Institute of Technology,  SE-100 44 Stockholm, Sweden. Email:~frachon@kth.se, erikni6@kth.se, sara.zahedi@math.kth.se.}, Erik Nilsson\footnotemark[1], Sara Zahedi\footnotemark[1]}

\date{}
 
\begin{document} 
\maketitle
\begin{abstract}
 % ABSTRACT
We extend the divergence preserving cut finite element method presented in [T. Frachon, P. Hansbo, E. Nilsson, S. Zahedi, SIAM J. Sci. Comput., 46 (2024)] for the Darcy interface problem to unfitted outer boundaries. We impose essential boundary conditions on unfitted meshes with a stabilized Lagrange multiplier method. The stabilization term for the Lagrange multiplier is important for stability but it may perturb the approximate solution at the boundary. We study different stabilization terms from cut finite element discretizations of surface partial differential equations and trace finite element methods. To reduce the perturbation we use a Lagrange multiplier space of higher polynomial degree compared to previous work on unfitted discretizations.  We propose a symmetric method that results in 1) optimal rates of convergence for the approximate velocity and pressure;  2) well-posed linear systems where the condition number of the system matrix scales as for fitted finite element discretizations; 3) optimal approximation of the divergence with pointwise divergence-free approximations of solenoidal velocity fields. The three properties are proven to hold for the lowest order discretization and numerical experiments indicate that these properties continue to hold also when higher order elements are used.
\end{abstract}


\maketitle  

\section{Introduction and the mathematical model}\label{sec: intro}
Satisfying the incompressibility condition pointwise has been the focus of the recent developments of unfitted discretizations of partial differential equations modeling the dynamics of incompressible flows~\cite{LiuNeiOls23, FraHaNilZa22, BurHanLarst24, FraNilZa23, Lehrenfeld2023Divfree}.

It is well-known that based on how the unfitted boundary cuts through the computational mesh unfitted finite element methods may produce ill-conditioned linear system matrices. In cut finite element methods (CutFEM) a common strategy to remedy this ill-conditioning has been to add ghost penalty stabilization terms to the weak formulation~\cite{Bu10, BuHa12}. These stabilization terms are sometimes also needed to ensure that the discretization is stable independent of how the unfitted boundary is positioned relative the mesh. In~\cite{LiuNeiOls23} a CutFEM for the Stokes equations based on divergence-free elements is proposed but due to the standard ghost penalty stabilization the velocity is pointwise divergence-free only outside a band around the unfitted boundary. In~\cite{FraHaNilZa22}, the standard ghost penalty stabilization terms~\cite{HaLaZa14, MaLaLoRo14} are slightly modified to preserve the divergence-free property of the underlying finite elements. The discretization in~\cite{FraHaNilZa22} of the Darcy interface problem is unfitted with respect to an interface (an internal boundary) but fitted with respect to the domain boundary and essential boundary conditions are imposed strongly. 

In~\cite{FraNilZa23} weak imposition of Dirichlet boundary conditions on unfitted boundaries, in connection with the Stokes equations, is studied and it is illustrated with numerical experiments that the pointwise divergence-free property as well as pressure robustness~\cite{linke_pressure-robustness_2016} can be affected if the boundary condition is not satisfied accurately. In~\cite{BurHanLarst24} a first order cut finite element discretization of the Stokes equations based on a stabilized Lagrange multiplier method for enforcing Dirichlet boundary conditions is proposed. The method preserves the divergence-free condition pointwise. However, a numerical experiment in the paper illustrates that errors at the boundary destroy pressure robustness. We illustrate  with a similar numerical experiment that the imposition of the essential boundary condition can be improved by using a Lagrange multiplier space of higher polynomial degree than used in~\cite{BurHanLarst24}. We analyze this method for the Darcy problem, extending the divergence preserving method in~\cite{FraHaNilZa22} to handle  Dirichlet boundary conditions on unfitted boundaries.

We consider an open bounded convex domain $\Omega \subset \RR^d$, $d=2,3$, and assume $\Omega$ has a piecewise smooth closed simply connected boundary $\partial \Omega$. Given source terms $\pmb{f}\in \mathbf{L}^2(\Omega),\ g\in L^2(\Omega)$, boundary data $\uB \in H^{-1/2}(\partial\Omega)$, the inverse permeability $\permi$, we seek fluid velocity $\pmb{u}: \Omega \rightarrow \RR^d$ and pressure $p: \Omega \rightarrow \RR$ with prescribed average $\bar{p}\in \RR$, satisfying the following Darcy problem also called the mixed formulation of the Poisson problem: 
\begin{subequations}\label{eqs:strongDarcy}
\begin{alignat}{2}
  \permi \pmb{u} + \nabla p & = \pmb{f} 
 && \text{ in } \Omega, \label{eq:pressure} \\
  \hfill \dive\pmb{u} & =  g
  &&\text{ in } \Omega, \label{eq:conservation} \\
  \hfill \pmb{u}\cdot\pmb{n} & =  \uB
  &&\text{ on } \partial\Omega, \label{eq:BCu} \\
  \int_\Omega p &= \bar{p},
\end{alignat}
\end{subequations}
where $\pmb{n}$ is the outward unit normal vector to $\partial\Omega$. 
We assume that $\permi=\bfK^{-1}$ with the permeability tensor $\bfK: \Omega \rightarrow \RR^d$ being such that $\bfK$ is diagonal with the smallest diagonal entry being positive and bounded away from zero.  The boundary condition \eqref{eq:BCu} is an essential (Dirichlet) condition on the normal flux.

We use the Raviart-Thomas spaces as approximation spaces for the velocity and utilize the stabilization terms proposed in~\cite{FraHaNilZa22}. We show that the stabilized Lagrange multiplier method we propose for imposing the essential boundary condition on unfitted meshes does not destroy the divergence-free property but the stabilization may purturb the solution at the boundary. Compared to previous work we use Lagrange multiplier spaces of higher polynomial degree, the polynomial degree matches the velocity space. We illustrate numerically that this choice improves the imposition of the boundary condition which can be of importance in some applications. For the Lagrange multiplier we also study and compare different stabilization terms such as those in the literature of cut finite element discretizations of surface partial differential equations and trace finite element methods~\cite{LaZa2020SurfStab, BurHanLarMas2018MFD, GrandeLehr2018Surf}. We present a symmetric unfitted discretization for the Darcy problem which is of optimal convergence order, results in well-conditioned linear systems, and preserves the incompressibility condition pointwise.  We prove that these properties hold for the lowest order elements in two and three space dimensions. In the numerical experiments we also consider higher order elements in two space dimensions.

The paper is organized as follows.
In Section \ref{sec:Nummeth} we present a cut finite element method for the Darcy problem. In Section \ref{sec:analysis} we analyze the method. We derive a priori error estimates in Section \ref{sec:apriori}, and we prove and estimate how the condition number scales with mesh size in Section \ref{sec:condnbr}. 
In Section \ref{sec:numex} we present numerical experiments which test and confirm the theoretical results. 
In Section \ref{sec:conclusion} we summarize our findings. 
Appendix A  contains auxiliary lemmas used for proving the inf-sup condition. Appendix B is devoted to techniques that improve the implementation and efficiency of the stabilization terms.

\section{A cut finite element method}\label{sec:Nummeth}
In this section we introduce the mesh, the finite element spaces, and the weak form, which together define our unfitted discretization of the Darcy problem~\eqref{eqs:strongDarcy}. For the ease of presentation we assume up to Section \ref{sec:numex} that $\bar{p}=0$. 

\subsection{Mesh}
Let $\Omega_0$ be a polytopal open subset of $\RR^d$, $d=2,3$, such that  $\Omega\subset\Omega_0$. We assume that $\partial\Omega \cap \partial\Omega_0 = \emptyset$. Let $\{ \mcT_{0,h} \}_h$ be a quasi-uniform family of simplicial meshes of $\Omega_0$ with the subscript $h$ being the piecewise constant function that on an element $T$ is equal to $h_T$, the diameter of $T$. We assume that $h_0=\max_{T\in\mcT_{0,h}} h_T<<1$. We have 
\begin{align}
	\Omega_0 = \bigcup_{T\in\mcT_{0,h}}T.
\end{align}	

We call $\mcT_{0,h}$ the background mesh, and from it we construct the active mesh
\begin{align}
	\mesh := \left\{ T\in\mcT_{0,h}: |T\cap\Omega|>0 \right\}
\end{align}
on which we define our finite element spaces. The computational domain will be given by
\begin{align}
	\Omdh := \bigcup_{T\in\mesh}T,\quad \text{ and } \quad \Omega\subset\Omdh\subset\Omega_0.
\end{align}	

We denote the boundary by $\Sigma$: 
\begin{align}\label{eq:sigma}
	\Sigma:=\partial\Omega.
\end{align}
Let $\mcT_\Sigma$ be the set of elements intersected by $\Sigma$ and let $\Omsig$ be the corresponding domain:
\begin{align}
	\mcT_\Sigma &:= \left\{T \in \mcT_h : | \Sigma \cap T| > 0 \right\}, \\ 
	\Omsig &:= \bigcup_{T\in\mcT_{\Sigma}}T.
\end{align}
We use $\Omsig\subset \Omdh$ to talk about the subset of the computational domain corresponding to $\mcT_\Sigma$.

Given some submesh $\mcT'_h\subseteq\mesh$ we denote by $\mcF(\mcT'_h)$ the set of faces in $\mcT'_h$. By $\mcF^{\circ}(\mcT'_h)$ we denote the set of interior faces in $\mcF(\mcT'_h)$; that is faces $F=T_1\cap T_2$ for some $T_1,T_2\in\mcT'_h$. 
We will need the following two sets of faces:
%on which we will apply stabilization terms:
\begin{align}
	\mcFs &:= \left\{ F\in\mcF(\mcT_\Sigma): |F\cap\Omega|>0 \right\}, \\
	\mcFsi &:= \mcF^{\circ}(\mcT_\Sigma).
\end{align}
See Figure \ref{fig:static active mesh} for an illustration of the sets and domains introduced in this section. 
\begin{figure}[h!]
	\centering
	\begin{subfigure}[b]{0.45 \textwidth}  	 	 	
	\centering	
	\includegraphics[scale=0.2]{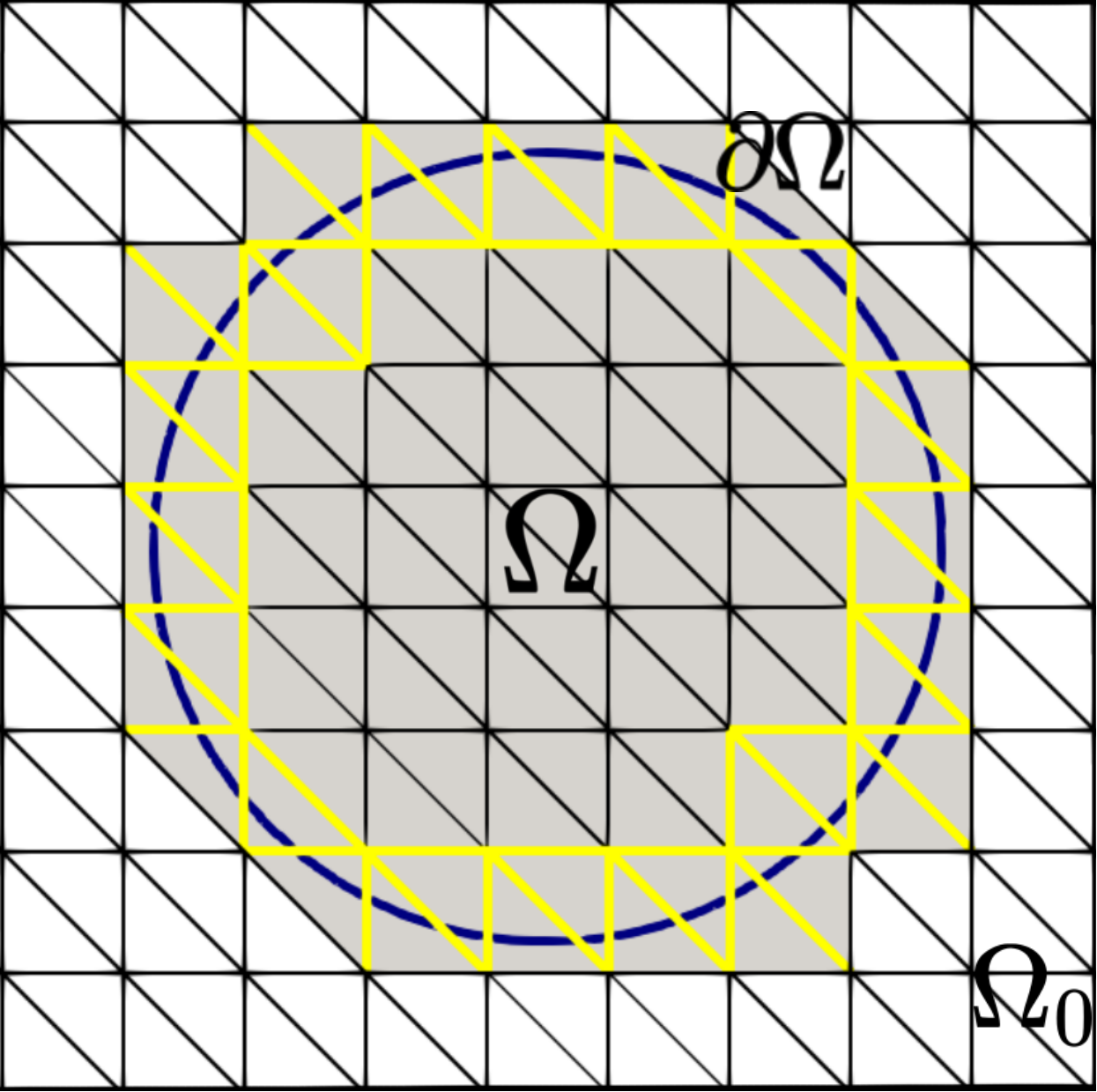}
	\end{subfigure}
        \caption{Illustration of the domains $\Omega\subset\Omdh\subset\Omega_0$, the active mesh $\mesh$ (triangles in grey), and faces in $\mcF_{\Sigma}$ (faces marked in yellow). The gray colored domain is $\Omdh$.}  
        \label{fig:static active mesh}      	
\end{figure}

% For $h_F=\text{diam}(F)$, $F\in\mcF(\mcT_{h})$, we also assume 
% \begin{align}
% 	h_T \eqsim h_F.
% \end{align}

%For two quantities $x,y$ we will write $x\lesssim y$ if and only if $x\leq C y$ for some constant $C$ that is independent of the mesh parameter $h$. Similar meaning is understood for $x\gtrsim y$. If both $x\lesssim y$ and $x\gtrsim y$ hold, we write $x\simeq y$.

\subsection{Spaces}
Let $U\subset\RR^d,\ d=2,3$. For $p,q:U\to \RR$ scalar functions
%Similar notation hold for vector-valued functions. 
we define the $L^2$-inner product and $L^2$-norm over $U$ by 
\begin{align}
(p,q)_{U} := \int_{U} pq,\quad 
	\|p\|_{U} := (p,p)_U^{1/2}.
\end{align}
Similar notation hold for vector-valued functions. Let $B\subset\partial U$ be nonempty. In addition to $L^2(U)$ we will consider the Sobolev spaces
\begin{alignat}{2}
	\bfH^{\dive}(U) &:= \left\{\bfv \in  \mathbf{L}^2(U) : \ \dive \bfv \in L^2(U)\right\}, \quad &&\bfH^{\dive}_{w}(U):= \left\{ \bfv\in \bfH^{\dive}(U) : \bfv\cdot\bfn\lvert_\Sigma= w \in H^{-1/2}(\Sigma) \right\}, \\
  H^1(U)&:= \left\{ p\in L^2(U) : \nabla p\in \mathbf{L}^2(U)\right\},\quad &&H^1_{0,B}(U):= \left\{ p\in H^1(U) : p\lvert_B= 0 \right\}, \\
  \bfH^1(U)&:= \left\{ \bfv\in \mathbf{L}^2(U) : \nabla \bfv\in \mathbf{L}^2(U)\right\},\quad &&\bfH^1_{0,B}(U):= \left\{ \bfv\in\bfH^1(U) : \bfv\lvert_B= 0 \right\}.
\end{alignat}
%We will take the convention that $H^0(U)=L^2(U).$ 
The respective norms are
\begin{align}
	\| \bfv \|_{\bfH^{\dive}(U)} &:= (\| \bfv \|_{U}^2+\| \dive \bfv \|_{U}^2)^{1/2}, \\
	\| p \|_{H^1(U)} &:= (\| p \|_{U}^2+\| \nabla p \|_{U}^2)^{1/2}, \\
	\| \bfv \|_{\bfH^1(U)} &:= (\| \bfv \|_{U}^2+\| \nabla \bfv \|_{U}^2)^{1/2}.
\end{align}
In the a priori estimates we will use the space 
  \begin{equation}
  \Hdivr(U):= \left\{\pmb{v}:\ \pmb{v} \in  \mathbf{H}^1(U),  \ \dive \pmb{v} \in H^1(U)\right\}, %\quad r \in \NN_0
    \end{equation}
   with the norm
  \begin{equation}
  \| \bfv \|_{\Hdivr(U)}^2 = \| \bfv \|_{\bfH^1(U)}^2+\| \dive \bfv \|_{H^1(U)}^2. 
   \end{equation}

Given $T\in\mesh$ let $\RT_{k}(T)$ denote the local Raviart-Thomas space \cite{BofBreFor13} of order $k\geq 0$. % and let $\BDM_k$ denote the Brezzi-Douglas-Marini space of order $k\geq 1.$
Define the following finite element spaces on the mesh $\mcT_{h}$:
\begin{align} \label{eq:Vi}
	\Vk & :=\left\{\bfv_h\in \bfH^{\dive}(\Omdh): \ \bfv_h\lvert_T\ \in \RT_k(T), \ \forall T\in \mcT_{h} \right\}, \\
  	\Qk& := \left(\bigoplus_{T \in \mcT_{h}} Q_k(T) \right) / \RR \\ 
	&= \left\{ q_h \in L^2(\Omdh) : \ q_h\lvert_T\ \in Q_{k}(T) \ \ \forall T\in \mcT_{h}, \ \int_{\Omega}q_h = 0 \right\},
\end{align}
where $Q_{k}(T)$ is the space of polynomials in $T$ of degree less or equal to $k$. %We write $\int_{\Omdh} q_h = \bar{q}_h$ as a short hand for the integral over $\Omdh$ of $q_h$. 
The following commuting diagram holds (see \cite[Equation (2.5.27)]{BofBreFor13}): 
\begin{equation}\label{eq:commuting_diagram}
	\begin{tikzcd}
		\bfH^{\dive}(\Omdh)\cap \prod_{T\in\mesh} \bfH^{1}(T) \arrow{r}{\dive} \arrow[swap]{d}{\interpv} & L^2(\Omdh) \arrow{d}{\projp} \\%
		\Vk \arrow{r}{\dive} & \Qk
	\end{tikzcd}
\end{equation}
Here, $\projp:L^2(\Omdh)\to \Qk$ is the $L^2$-projection and $\interpv : \bfH^{\dive}(\Omdh)\cap \prod_{T\in\mesh} \bfH^{1}(T) 
\to \Vk$ is the global interpolation operator associated to $\Vk$.

To impose essential boundary conditions we will use the following space of discrete Lagrange multipliers: 
\begin{align}
	\Qks& :=\bigoplus_{T \in \mcT_\Sigma} Q_{k+1}(T)
\end{align}
consisting of piecewise polynomials of degree $k+1$ on the elements intersected by the boundary $\Sigma$. See Remark \ref{rem:lagrange_why_kp1} for details on why we choose the degree to be $k+1$ instead of $k$.

\subsection{The weak formulation}\label{sec:method}
A weak formulation to \eqref{eqs:strongDarcy} which enforces the essential boundary condition on $\Sigma$ through an additional equation using a Lagrange multiplier reads as follows: Find $(\bfu,p,\phi)\in \bfH^{\dive}_{u_B}(\Omega)\times H^1(\Omega)/\RR \times H^{1/2}(\Sigma)$ such that 
\begin{align}\label{eq:weakdarcy}
  m(\bfu,\bfv) + b(\bfv,p) + c(\bfv,\phi) &= (\pmb{f},\bfv)_\Omega%\FF (\bfv) 
  &\text{ for all } \bfv\in \bfH^{\dive}_{0}(\Omega), \\
  b(\bfu,q) &= -(g,q)_{\Omega} &\text{ for all } q\in H^1(\Omega)/\RR, \nonumber \\
  c(\bfu,\chi) &= (u_B,\chi)_{\Sigma} &\text{ for all } \chi\in H^{1/2}(\Sigma), \nonumber
\end{align}
where, 
\begin{align}
  m(\bfu,\bfv) &:= (\permi \bfu,\bfv)_{\Omega},\label{eq:a}\\
  b(\bfu,q) &:= -(\dive \bfu,q)_{\Omega}, \label{eq:b} \\
  c(\bfu,\chi) &:= (\bfu\cdot\bfn,\chi)_{\Sigma}, \label{eq:c} %\\
\end{align}
The integrals on the boundary $\Sigma$ are well defined since 
$\uB \in H^{-1/2}(\boundu).$ We shall assume that the data is regular enough so that the solution $(\bfu,p)$ lies inside $\bfH^{1,\dive}(\Omega)\times H^1(\Omega)$. An example is given by $u_B=0$, $\pmb{f}\in \bfH^2(\Omega)$ and $g\in H^1(\Omega)$, c.f. \cite[Thm. 6.3.5]{evans2022pde}. 

We propose the following cut finite element discretization:  Find $(\bfu_h,p_h,\phi_h)\in \Vk \times \Qk \times \Qks$ such that 
\begin{align}\label{eq:discretedarcy2}
	\Mf (\bfu_h,\bfv_h) + \Bf(\bfv_h,p_h) + c(\bfv_h,\phi_h) &= (\pmb{f},\bfv)_\Omega%\FF (\bfv_h) 
	&\text{ for all } \bfv_h\in \Vk, \\
	\Bf(\bfu_h,q_h) &= -(g,q_h)_{\Omega} &\text{ for all } q_h\in \Qk, \nonumber \\
	c(\bfu_h,\chi_h)-s_c(\phi_h,\chi_h) &= (u_B,\chi_h)_{\Sigma}  &\text{ for all } \chi_h\in \Qks. \nonumber
\end{align}
Here 
\begin{align}
	\Mf (\bfu_h,\bfv_h) &:= m(\bfu_h,\bfv_h) + s(\bfu_h,\bfv_h) =(\permi \bfu_h,\bfv_h)_{\Omega} + s(\bfu_h,\bfv_h), \label{eq:A}\\
	\Bf(\bfu_h,q_h) &:=b(\bfu_h,q_h)-s_b(\bfu_h,q_h) = -(\dive \bfu_h,q_h)_{\Omega} - s_b(\bfu_h,q_h), \label{eq:B} \\
	c(\bfu_h,\chi_h) &:= (\bfu_h\cdot\bfn,\chi_h)_{\Sigma}, \label{eq:C}
\end{align}
with the ghost penalty stabilization terms connecting elements with a small intersection with $\Omega$ to elements with a large intersection defined as
\begin{align}
	s(\bfu_h,\bfv_h) &:=  \sum_{F \in \mcFs}\sum_{j=0}^{k+1}
	\tau h^{2j+1}  (\jump{D^j_{\bfn_F} \bfu_{h} }, \jump{D^j_{\bfn_F} \bfv_{h}})_{F}, \label{eq:full stab-u}\\
	s_b(\bfu_h,q_h) &:=  \sum_{F \in \mcFs}\sum_{j=0}^{k}
	\tau_b h^{2j+1}  ( \jump{D_{\bfn_F}^j (\dive \bfu_{h})}, \jump{D_{\bfn_F}^j q_{h}} )_{F}, \label{eq:full stab-b}\\
	s_c(\phi_h,\chi_h) &:=  \sum_{F \in \mcFsi}\sum_{j=0}^{k+1}
	\tau_c h^{2j-1} ( \jump{D^j \phi_h}, \jump{D^j \chi_h} )_F \label{eq:full stab-c} \\
	&\quad +\sum_{j=1}^{k+1}
   \tau_c h^{2j-1} ( D_{\bfn}^j \phi_h, D_{\bfn}^j \chi_h )_\Sigma. \nonumber
   % \quad\quad \gamma\in [-1,1]. \nonumber
\end{align}
We define the jump operator across a face $F\in \mcFs$ by
\begin{equation}
	\jump{v}(\bfx) := \lim_{\epsilon\to 0^+} (v(\bfx+\epsilon\bfn_F) - v(\bfx-\epsilon\bfn_F)), \label{eq:jump}
\end{equation}
where $\bfn_F$ is the normal vector associated to $F$ and $\bfx\in\Omdh$. For a vector valued function $\bfv$ we define $[\bfv]$ component-wise as \eqref{eq:jump}. 
Here $D^j_{\bfn_F} \bfv_h$ denotes the normal derivative of order $j$ across the face $F$, with $\jump{D^0_{\bfn_F} \bfv_h} = \jump{\bfv_h}$. The full derivative $D^j \phi_h$ also includes the tangential derivative(s). 
The stabilization parameters $\tau,\tau_{b},\tau_c$ are positive constants. 

\begin{remark}[Stabilization]
  The stabilization terms extend the control (measured in a suitable norm) of finite element functions from the physical domain to the active domain, see Lemma~\ref{lem:sp_ineq} and~\ref{lem:lagrange}. The stabilization term $s$ extends the control in the standard $L^2$-norm. It is important to stabilize the bilinear form  $b(\bfu_h,q_h)$ correctly in order to not destroy the divergence preserving property of the finite element pair $(\Vk,\Qk)$. Therefore, we proposed in \cite{FraHaNilZa22} to stabilize the bilinear form $b$ using $s_b(\bfu_h,q_h)$. Note that $s_b(\bfu_h,q_h)=s_k(\dive \bfu,q_h)$,  where 
$s_k$ is the standard ghost penalty term, used in earlier work to stabilize the pressure variable, 
\begin{align}
	s_k(p_h,q_h) &:=  \sum_{F \in \mcFs}\sum_{j=0}^{k}
	\tau h^{2j+1}  ( \jump{D^j_{\bfn_F}p_h}, \jump{D^j_{\bfn_F} q_h} )_{F} \label{eq:full stab-p}.
        \end{align}
The sum goes to $k$ since $q_h, \dive \bfu \in \Qk$ while $\bfu, \bfv \in \Vk$ and therefore the sum in $s(\bfu, \bfv)$ goes to $k+1$. We will when it is clear leave out the subscript $k$.
In order to avoid evaluation of derivatives in the implementation of the method we switch the face stabilization terms to the equivalent extension based patch stabilization of~\cite{preuss2018}, see also Appendix \ref{sec:patch-stab}. Then, the stabilization term does not explicitly depend on the polynomial order $k$ and the difference in the stabilization of $\Mf$ and $\Bf$ is only in the argument to $s$. It is sufficient to stabilize over a macroelement partition \cite{LarZah23} of $\Omega_{\mesh}$. This means that stabilization can be applied on a subset of $\mcFs$ leading to increased sparsity in the resulting system matrix and a solution that is less sensitive to the choice of the stabilization constant.  See Appendix \ref{sec:macro-elements} for details on this. 
\end{remark}

\begin{remark}[The approximation space for the discrete Lagrange multiplier] \label{rem:lagrange_why_kp1}
  Let $k=0$, for $\bfv_h\in \RT_0(T)$ and $\Sigma \cap T$ being a subset of a hyperplane for $T \in \mcT_\Sigma$ we have that $\bfv_h\cdot\bfn|_{\Sigma\cap T} \in Q_{0}(\Sigma\cap T)$. For $\chi_h, \phi_h \in Q_{h,0}^{\Sigma}$ the stabilization term $s_c(\phi_h, \chi_h)=\tau_ch^{-1}(\jump{\phi_h}, \jump{\chi_h} )_F$ in equation \eqref{eq:discretedarcy2} is in general not zero (except for the globally constant functions). Thus, the stabilization term $s_c$ perturbs the enforcement of the boundary condition and in general $c(\bfu_h,\chi_h) \neq (u_B,\chi_h)_{\Sigma}$. However, we will always have
\begin{equation}
  \int_{\Sigma} \bfu_h \cdot \bfn=\int_\Sigma u_B,
  \end{equation}
  which is a necessary condition for the method to be divergence preserving, see Section~\ref{sec:divpres}.  In the larger space $\QksZ$ there are more functions $\chi_h$ for which this perturbation due to the stabilization is zero or is small, for example all the continuous piecewise linear functions will only have the jump of the first order derivative left in the stabilization term.  This motivates choosing the approximation space for the Lagrange multipliers to be larger than $Q_{h,0}^{\Sigma}$. We emphasize that choosing $Q_{h,0}^{\Sigma}$, as the space for the Lagrange multipliers, still gives rise to a well-posed problem and optimal error estimates but errors at the boundary can in some cases be much worse than choosing the space $\QksZ$. This can be seen in the numerical example of Section~\ref{sec:alt stab}. 
  Note that for curved unfitted boundaries we no longer have $\bfv_h\cdot\bfn|_{\Sigma\cap T} \in Q_{0}(\Sigma\cap T)$ which also motivates the choice of a larger Lagrange multiplier space than $Q_{h,0}^{\Sigma}$. 
\end{remark}
\begin{remark}[The linear system] \label{rmk:saddle}% \textbf{A saddle-point problem} 
	The linear system associated with \eqref{eq:discretedarcy2} can be written as
	\begin{equation*}
		\begin{bmatrix}
			\mathbf{M} & \mathbf{B}^T & \mathbf{C}^T\\
			\mathbf{B} & \mathbf{0} & \mathbf{0} \\
			\mathbf{C} & \mathbf{0} & -\mathbf{S}_c
		\end{bmatrix}
		\begin{bmatrix}
			u \\
			p \\
			\phi
		\end{bmatrix} 
		= 
		\begin{bmatrix}
			F \\
			g \\
			u_B
		\end{bmatrix}
	\end{equation*}
	(Here $(u,p,\phi)$ represent degree of freedom values, and $(F,g,u_B)$ are the corresponding right hand side values.) 
	The weak imposition of the boundary condition on $\Sigma$ via the Lagrange multiplier $\chi$ casts the proposed method into the framework of perturbed composite saddle-point problems \cite[Ch. 3.5.4, Ch. 3.6]{BofBreFor13}, in contrast to \cite{FraHaNilZa22} where the method gives rise to a standard saddle-point problem. 
\end{remark}

\subsection{Alternative methods}
We present alternative stabilization terms to \eqref{eq:discretedarcy2} and study them in Section \ref{sec:numex}. We also compare the presented strategy, of using a stabilized Lagrange multiplier method, with prescribing the Dirichlet boundary condition using penalty.

\subsubsection{Alternative stabilization for the Lagrange multiplier}
An alternative to $s_c$ is to use the following stabilization terms proposed in~\cite{BurHanLarMas2018MFD,GrandeLehr2018Surf} in connection with TraceFEM and CutFEM for surface PDEs:
\begin{align}
	\tilde s_c(\phi_h,\chi_h) &:=  \sum_{F \in \mcFsi}
	\tau_c h^{-1} ( \jump{ \phi_h}, \jump{ \chi_h} )_F \nonumber\\
	&\quad + \tau_ch(D_{\bfn}\phi_h,D_{\bfn}\chi_h)_{\Omsig}. \label{eq:alt stab-c}
\end{align}	
Here the boundary term in $s_c$ (see \eqref{eq:full stab-c}) is replaced by a normal derivative term on $\Omsig$. The choice of the normal vector in the bulk term is not unique and in our test cases the choice for the normal vector affected the accuracy and the condition number. 

In the numerical experiments of Section \ref{sec:alt stab} we obtained better control of the condition number if all the derivatives in the face stabilization term are included or if the patch based stabilization is applied (see Appendix \ref{sec:patch-stab}). As such we propose the following alternative
\begin{align}
	\hat s_c(\phi_h,\chi_h) &:=  \sum_{F \in \mcFsi}\sum_{j=0}^{k+1}
	\tau_c h^{2j-1} ( \jump{D^j \phi_h}, \jump{D^j \chi_h} )_F \nonumber\\
	&\quad + \tau_ch(D_{\bfn}\phi_h,D_{\bfn}\chi_h)_{\Omsig} \label{eq:alt stab-c 2}.
\end{align}	
We compare \eqref{eq:alt stab-c 2} as an alternative to the stabilization \eqref{eq:full stab-c} in Example 1 of Section \ref{sec:numex}. We also compare \eqref{eq:alt stab-c} to \eqref{eq:alt stab-c 2} in the same subsection.

\subsubsection{A non-symmetric penalty formulation}\label{sec:penalty}
In the numerical experiments we compare with imposing the boundary condition weakly using a penalty parameter. 
The method we compare with reads as follows:
Find $(\bfu_h,p_h)\in \Vk \times \Qk$ such that 
\begin{align}\label{eq:discretedarcy_nitsche}
	\tilde \Mf (\bfu_h,\bfv_h) + \tilde\Bf(\bfv_h,p_h) &= (\pmb{f},\bfv)_\Omega + \lambda h^{-1}(u_B,\bfv_h\cdot\bfn)_\Sigma
	&\text{ for all } \bfv_h\in \Vk, \\
	\Bf(\bfu_h,q_h) &= -(g,q_h)_{\Omega} &\text{ for all } q_h\in \Qk . \nonumber
\end{align}
Here $\lambda\in\RR$ is a penalty parameter taken sufficiently large and 
\begin{align}
	\tilde \Mf (\bfu_h,\bfv_h) &:= \lambda h^{-1}(\bfu_h\cdot \bfn,\bfv_h\cdot\bfn)_\Sigma + \Mf (\bfu_h,\bfv_h), \\
	\tilde \Bf(\bfv_h,p_h) &:= (p,\bfv\cdot\bfn)_\Sigma + \Bf(\bfv_h,p_h).
\end{align}
Note that due to the term $(p,\bfv\cdot\bfn)_\Sigma$ in $\tilde \Bf$ this method is non-symmetric. This term is only present in case of Dirichlet boundary conditions. If one has Dirichlet boundary conditions on the entire boundary, the mean value of the pressure has to be prescribed, this is typically done via a Lagrange multiplier. In~\cite{FraNilZa23} we show that this Lagrange multiplier will perturb the divergence condition in the unfitted setting. We propose how to solve this problem in~\cite{FraNilZa23}. We do not include that discussion here. In this paper we compare the proposed scheme to this penalty method in Example 3 of Section \ref{sec:numex}, where we have mixed boundary conditions.

\section{Analysis of the Method}\label{sec:analysis}
The analysis of the method is for the lowest order case, i.e., $k=0$. Here we assume that $\Omega$ is such that $\Sigma\cap T$ is a linear segment in $\RR^2$ or a subset of a hyperplane in $\RR^3$ for every $T\in \mcT_\Sigma$ and all h<H<<1, where $H$ is an initial maximum mesh diameter. We utilize this assumption in Lemma~\ref{lem:dim_count}. 

For $\bfv\in \bfH^{1,\dive}(\Omdh)+\VkZ$, $q\in H^{1}(\Omdh)+\QkZ$, $\chi\in H^{2}(\Omsig)+\QksZ$, define the following norms:
\begin{align}
  \| \chi \|_h^2 &:=
   \|h^{-1/2} \chi \|_\Sigma^2 + s_c(\chi,\chi), \\
	\| q \|^2_{\Omdh} &:= \sum_{T\in\mesh} \| q \|^2_T, \\
	\tn \bfv \tn^2_\Mf &:= \Mf(\bfv,\bfv) 
	= \|\permi^{1/2}\bfv\|^2_{\Omega} + s(\bfv,\bfv), \\
	\tn \bfv \tn^2_h &:= \tn \bfv \tn_\Mf^2 + \| \dive \bfv\|_{\Omdh}^2 + 
	\|h^{1/2}\bfv\cdot \pmb{n}\|^2_{\Sigma}.
\end{align}
Let $D^{\ell}$ be the generalized derivative of order $\ell\geq 1$.
We let the standard seminorm of order $\ell$ over a domain  $U\subset\RR^d$ be denoted by 
\begin{align}
	|p|_{\ell,U}:=\|D^{\ell}p\|_{U}. 
\end{align}
We will many times apply the standard element-wise trace inequality \cite{BreSco} on $\partial T$ and the trace inequality \cite{HanHanLar2003TraceIneq} on $\Sigma \cap T$, $T \in \mesh$,
\begin{align}
  \|p\|_{\partial T} + \|p\|_{\Sigma \cap T} \lesssim h^{-1/2}\|p\|_T + h^{1/2}\|\nabla p\|_T, \quad &\forall p\in H^1(T). \label{eq:trace_ineq}
\end{align}
We will also use standard inverse inequalities~\cite{BreSco},
\begin{equation} \label{eq:inverse}
|p_h|_{j,T}\lesssim h^{s-j}\|p_h \|_{H^s(T)},  \quad 0\leq s \leq j, \ p_h\in Q_\ell(T),\ \ell\geq 0,
\end{equation}
which together with the trace inequality above yields 
  \begin{align}
	\|p_h\|_{\partial T} + \|p_h\|_{\Sigma \cap T} \lesssim h^{-1/2}\|p_h\|_T, \quad &\forall p_h\in Q_\ell(T), \ \ell\geq 0.\label{eq:discrete_trace_ineq}
\end{align}

We start by showing that $\|\cdot\|_{\Omdh}$ and $\|\cdot\|_{\Omega} + s_k(\cdot,\cdot)^{1/2}$ are equivalent norms on $\Qk$, all bilinear forms are continuous and use this result to show that $h^{-1}\|\cdot \|_{\Omsig}$ and $h^{-1/2}\|\cdot\|_{\Sigma}+s_c(\cdot,\cdot)^{1/2} $ are equivalent norms on $\QksZ$.

\begin{lemma}\textbf{(Equivalent norms)}\label{lem:sp_ineq}
	Let $q_h\in\Qk$. The following inequalities hold
	\[ \|q_{h} \|^2_{\Omdh} \lesssim \|q_{h}\|^2_{\Omega} + s_k(q_{h},q_{h}) \lesssim \|q_h\|^2_{\Omdh}. \]
\end{lemma}
\begin{proof}
	For the leftmost inequality see \cite[Lemma 3.8]{HaLaZa14} or \cite[Lemma 5.1]{MaLaLoRo14}. For $q_h\in\Qk$, the standard element-wise trace inequality followed by the standard inverse inequality gives
	\begin{align}%\label{eq:stabbound}
		s_k(q_{h},q_{h}) &%=
		%\sum_{F \in \mcFs} \tau_p h \|\jump{q_{h}}\|^2_{F}
                \lesssim \sum_{T\in\mcT_{h}} \| q_h \|_{T}^2.
	\end{align}
\end{proof}
A similar proof can be applied component-wise to show that
\begin{align}\label{eq:sa_stabbound}
\|\bfv_h\|^2_{\Omdh} \lesssim	\|\bfv_h\|^2_\Omega+s(\bfv_h,\bfv_h) &\lesssim \|\bfv_h\|^2_{\Omdh}, \quad\quad \bfv_h\in\Vk.
\end{align}

\begin{lemma}\textbf{(Continuity)}\label{lem:continuity}
	The bilinear forms are continuous; 
	\begin{align}
		\Mf(\bfu,\bfv) &\lesssim \tn \bfu \tn_h \tn \bfv \tn_h,& \forall \bfu,\bfv\in \left(\bfH^{1,\dive}(\Omdh)+\VkZ\right), \\
		\Bf(\bfv,q) &\lesssim \tn \bfv \tn_h\|q\|_{\Omdh},& \forall \bfv\in \left(\bfH^{1,\dive}(\Omdh)+\VkZ\right),q\in \left(H^{1}(\Omdh)+\QkZ\right),\label{eq:contB}\\
		c(\bfv,\chi) &\lesssim \|h^{1/2}\bfv\cdot\bfn\|_{\Sigma}\|h^{-1/2}\chi\|_\Sigma,
                %   \tn \bfv \tn_h\|\chi_h\|_{h},
                & \forall \bfv\in \left(\bfH^{1,\dive}(\Omdh)+\VkZ\right),\chi \in \left(H^{2}(\Omsig)+\QksZ\right), \\
		s_c(\phi,\chi) &\lesssim \|h^{-1}\phi\|_{\Omsig}\|h^{-1}\chi\|_{\Omsig},
                % \|\phi\|_{h}\|\chi\|_{h},
                & \forall \phi\in \left(H^{2}(\Omsig)+\QksZ\right),\chi\in \left(H^{2}(\Omsig)+\QksZ\right).\label{eq:contsc}
	\end{align}
\end{lemma}
\begin{proof}
Applying Cauchy-Schwartz we see that $\Mf(\bfu,\bfv) \lesssim \Mf(\bfu,\bfu)^{1/2}\Mf(\bfv,\bfv)^{1/2}$ whereby the first inequality follows.
For $\Bf$ we also apply Cauchy-Schwartz to get
	\begin{align}
          \Bf(\bfv,q) &= -(\dive\bfv,q)_{\Omega}-s_b(\bfv,q) \lesssim \| \dive\bfv \|_{\Omdh}\|q\|_{\Omdh} + s_k(\dive\bfv,\dive\bfv)^{1/2}s_k(q,q)^{1/2}.
	\end{align}
        Note that for functions $\bfv\in \bfH^{1,\dive}(\Omdh)$ or $q \in \bfH^{1}(\Omdh)$ the term $s_k(\dive\bfv,\dive\bfv) s_k(q,q)$ with $k=0$ is zero and we have shown that $\Bf$ is continuous. For $q,\dive\bfv \in \QkZ$ we can apply Lemma \ref{lem:sp_ineq} to get
       \begin{align}
          \Bf(\bfv,q) 
		&\lesssim \| \dive\bfv\|_{\Omdh}\|q\|_{\Omdh}  \lesssim \tn\bfv\tn_h\|q\|_{\Omdh}.
	\end{align} 
        Applying the Cauchy-Schwartz inequality we directly get
	\begin{align}
		c(\bfv,\chi) &= (\bfv\cdot\bfn,\chi)_{\Sigma} 
		\lesssim \|h^{1/2}\bfv\cdot\bfn\|_{\Sigma}\|h^{-1/2}\chi\|_\Sigma.
           	\end{align}

Finally, the stabilization term $s_c$ is zero for functions in $H^{2}(\Omsig)$ and for $\phi, \chi \in \QksZ$ we apply Cauchy-Schwartz, the inequality \eqref{eq:discrete_trace_ineq}, and a standard inverse inequality to get  
	\begin{align}
		s_c(\phi,\chi) &= \sum_{F \in \mcFsi} \tau_c\left(
			h^{-1} \|\jump{\phi}\|_F\|\jump{\chi}\|_F + \tau_c h \|\jump{\nabla\phi}\|_F\|\jump{\nabla\chi}\|_F\right) +h\|\nabla\phi \cdot \bfn\|_{\Sigma}\|\nabla\chi \cdot \bfn\|_{\Sigma}  \\
		&\lesssim \sum_{T\in\Omega_\Sigma} \left( h^{-2} \|\phi\|_T \|\chi\|_T +  \|\nabla\phi\|_T \|\nabla\chi\|_T \right) \nonumber\\
		%&\lesssim \|h^{-1/2}\phi\|_{\Omsig}\|h^{-1/2}\chi\|_{\Omsig}.
		        &\lesssim \|h^{-1}\phi\|_{\Omsig}\|h^{-1}\chi\|_{\Omsig}
                                           \nonumber
	\end{align}
      
\end{proof}

\begin{lemma}\label{lem:lagrange}
  The following inequality holds for any $\chi_h\in\QksZ$, 
	\begin{align}\label{eq:Z_control}
	 \|h^{-1}\chi_h\|_{\Omsig}^2 
	  \lesssim \|h^{-1/2} \chi_h\|^2_{\Sigma}+s_c(\chi_h,\chi_h)\lesssim
          \|h^{-1}\chi_h\|_{\Omsig}^2.
	\end{align}
\end{lemma}

\begin{proof}
  From \cite[Lemma 3.2]{LarZah23} we have 
  \begin{align}
         \|\chi_h\|^2_{\Omsig} 
		\lesssim h \|\chi_h\|^2_{\Sigma}+h^2s_c(\chi_h,\chi_h).
  \end{align}
Using also inequality \eqref{eq:discrete_trace_ineq} and the continuity of $s_c$~\eqref{eq:contsc}  we obtain the desired estimate
\begin{align}
h^{-2}\|\chi_h\|^2_{\Omsig} 
\lesssim h^{-1} \|\chi_h\|^2_{\Sigma}+s_c(\chi_h,\chi_h) \lesssim h^{-2}\|\chi_h\|^2_{\Omsig}.
 \end{align}
 \end{proof}

Before we prove stability and derive a priori error estimates we show a result on the divergence preserving property of the proposed scheme. 

\subsection{Divergence preserving properties}\label{sec:divpres}
That the proposed method produces pointwise divergence-free approximations of solenoidal velocity fields (the case $g=0$) follows from the proof in \cite[Theorem 3.2]{FraHaNilZa22}. The proposed method will not satisfy $\dive \bfu_h = g$ pointwise for a general $g$. However, the condition is satisfied pointwise also for nonzero functions $g$ for which there is an extension $E_{\Omdh}g$ for which $\dive\bfu_h-E_{\Omdh}g  \in \Qk$ and $s_k(E_{\Omdh}g, E_{\Omdh}g)=0$. For an example see Theorem~\ref{thm:divpres}. 
For a local mass conservation which always holds we refer the reader to the end of Appendix \ref{sec:macro-elements}. 
% not even if $Eg\in \Qk$.

Let $\Pk \subset H^{1}(\Omdh)$ be the space of continuous Lagrange polynomials of order $k\in\{0,1\}$ defined on $\Omdh$. Consider the restriction to $\Omega$:
\begin{align}
	\Pk|_\Omega=\left\{q_h|_\Omega : q_h\in\Pk\right\}\subset H^1(\Omega)
\end{align}
and let $E_{\Omdh} : \Pk|_{\Omega} \to \Pk$ be the canonical polynomial extension operator such that $E_{\Omdh}q|_{\Omega}=q$ for $q\in\Pk(\Omega).$
\begin{thm}\textbf{(The divergence-preserving property)}\label{thm:divpres}
	Let $g$ be such that $E_{\Omdh}g\in\Pk$ and assume $\bfu_h\in\Vk$ satisfies \eqref{eq:discretedarcy2}. Then $\dive\bfu_h = g$ in $\Omega$.
\end{thm}
\begin{proof}
	We have that $\bfu_h\in\Vk$ satisfies 
	\begin{align*}
		0 = \int_{\Omega} (\dive\bfu_h -g)q_h + \sum_{F \in \mcFs}\sum_{j=0}^{k}
		\tau_b h^{2j+1}  ( \jump{D_{\bfn_F}^j (\dive \bfu_{h})}, \jump{D_{\bfn_F}^j q_{h}} )_{F},  \quad \forall  q_h\in\Qk.
	\end{align*}
	We may choose $q_h=\dive\bfu_h-E_{\Omdh}g$. Note that $\int_\Omega q_h=\int_{\Omega} \dive\bfu_h- \dive\bfu=\int_{\Sigma} \bfu_h \cdot \bfn- u_B=0$ hence $\dive \bfu_h-E_{\Omdh}g\ \in \Qk$.   
	Moreover, $\jump{E_{\Omdh}g}=0$ over faces $F\in\mcFs$ since $E_{\Omdh}g\in \Pk\subset H^{1}(\Omdh)$ and hence we can subtract $s_k(E_{\Omdh}g,E_{\Omdh}g)$. Finally, by Lemma \ref{lem:sp_ineq} we get
	\begin{equation}
		0=\|\dive\bfu_h-E_{\Omdh}g \|_{{\Omega}}^2+ s_k(\dive\bfu_h-E_{\Omdh}g, \dive\bfu_h-E_{\Omdh}g)
		\gtrsim \|\dive\bfu_h-E_{\Omdh}g \|_{\Omdh}^2\geq 0.
	\end{equation}
	 Thus, $\dive\bfu_h|_{\Omega} = E_{\Omdh}g|_{\Omega}=g$.
\end{proof}

\subsection{Stability}\label{sec:infsup}
We first prove an inf-sup result for $\Bf$ and then for the full system.
Let $\Omeps$ be a convex domain in $\RR^d$ such that $\Omdh \subset \Omeps$ for all $h<H$.
\begin{lemma} \textbf{(Inf-sup condition of $\Bf$)}\label{lem:B_infsup}
  For every $q_h\in \QkZ$ there exists $\bfv_h \in \Vk$ such that 
\begin{align}
\Bf(\bfv_h,q_h) \gtrsim \| q_{h} \|^2_{\Omdh}, \quad \tn \bfv_h \tn_h \lesssim \| q_h\|_{\Omdh}.
\label{eq:B_infsup} 
\end{align}
\end{lemma}
\begin{proof}
	Fix some $q_h \in \QkZ$. Let $q_{h}^e\in L^2(\Omeps)$ be the extension by $0$ of $q_h$ to $\Omeps$.
	We can apply Lemma \ref{lem:divergence_Poincare} to $-q_{h}^e$ to obtain a function $\bfv^e$ whose restriction to $\Omdh$, $\bfv = \bfv^e|_{\Omdh}$, lies inside the set 
	\begin{align}
		\{\bfv\in\bfH^1(\Omdh) : \dive\bfv = -q_h,\ \|\bfv\|_{\bfH^1(\Omdh)} \lesssim \|q_h\|_{\Omdh} \}\subset \bfH^{\dive}(\Omdh).
	\end{align} 
	(Note that the hidden constant in $\|\bfv\|_{\bfH^1(\Omdh)} \lesssim \|q_h\|_{\Omdh}$ does not depend on $h$ since the bound from Lemma \ref{lem:divergence_Poincare} is over $\Omeps$.)
	
	For $\bfv\in\bfH^{\dive}(\Omdh)$ consider the interpolant $\bfv_h=\pi_h\bfv\in\VkZ$. From the commuting diagram~\eqref{eq:commuting_diagram}, $\dive\bfv_h=\dive\bfv=-q_h$, and it follows that
	\begin{align*}
		\Bf(\bfv_{h},q_{h}) &= -(\dive \bfv_{h},q_{h})_{\Omega} - s_b(\bfv_{h},q_{h}) \\
		&= \|q_{h} \|_{\Omega}^2 + s_k(q_h, q_h) \gtrsim \| q_{h} \|^2_{\Omdh},
	\end{align*}
	where in the last inequality we used Lemma \ref{lem:sp_ineq}.
	Thus $\Bf(\bfv_{h},q_{h}) \gtrsim\| q_{h} \|^2_{\Omdh}  $ and we are done if we can show that $\tn \bfv_h \tn_h \lesssim \| q_h\|_{\Omdh}$. The interpolation operator is stable in the following sense, see \eqref{eq:stabilityintp}:
	\begin{align}\label{eq:vh_divstab}
		\| \bfv_h \|_{\bfH^{\dive}(\Omdh)} \lesssim \|\bfv\|_{\bfH^1(\Omdh)} \lesssim \|q_h\|_{\Omdh}.
	\end{align}
	Recall the definition of the norm $\tn \bfv_{h} \tn_h^2$,
	\begin{align*}
		\tn \bfv_{h} \tn_h^2 &= 
		\underbrace{ \|\permi^{1/2}\bfv_h\|^2_{\Omega}+s(\bfv_h,\bfv_h)
		+
		\|\dive\bfv_h\|^2_{\Omdh}}_{\mathbf{I}}+
		\underbrace{ \|h^{1/2}\bfv_h\cdot\bfn\|^2_{\Sigma} }_{\mathbf{II}}.
	\end{align*}
	\textbf{Term I}. 
	By inequalities \eqref{eq:sa_stabbound} and \eqref{eq:vh_divstab} we have 
	\begin{align}\nonumber
		\|\eta^{1/2}\bfv_h\|^2_{\Omega}+s(\bfv_{h},\bfv_{h})\lesssim \|\bfv_h\|^2_{\Omdh}+\|\dive\bfv_h\|^2_{\Omdh}\lesssim \|q_h\|^2_{\Omdh}.
	\end{align}
	\textbf{Term II}. 
	Inequality \eqref{eq:discrete_trace_ineq} yields
	\begin{align} \nonumber
		\|h^{1/2}\bfv_h\cdot \bfn\|_{\Sigma} \lesssim \|\bfv_h\|_{\Omdh} \lesssim \|\bfv_h\|_{\bfH^{\dive}(\Omdh)} \lesssim \|q_h\|_{\Omdh}. \nonumber
	\end{align}
	Combining the estimates for \textbf{Term I-II} 
	we get the desired estimate,  
	\begin{align*}
		\tn \bfv_{h} \tn_h &\lesssim  \| q_{h} \|_{\Omdh}. 
	\end{align*}
	
\end{proof}

We need the following lemma before we are ready to prove the stability of the method.
\begin{lemma}\label{lem:dim_count}
	For every $q_h\in \QkZ$ 
	and $\chi_h\in \QksZ$, 
	there exists a function $\bfv_h\in\VkZ$ 
	such that Lemma \ref{lem:B_infsup} is satisfied and 
	\begin{align}
		c(\bfv_h,\chi_h) &= \lambda, \quad \lambda\in\RR. \label{eq:A2}
	\end{align}
\end{lemma}
\begin{proof}	Without loss of generality we can take $\lambda=0$. 
  Let $\tilde \bfv_h\in \VkZ$ be the interpolant in the proof of Lemma~\ref{lem:B_infsup} attaining the inf-sup condition of $\Bf$. Note that for any $\bfv_h \in \VkZ$ (since $\Sigma\cap T$ is a linear segment or a subset of a hyperplane for every $T\in \mcT_\Sigma$) we have $\bfv_h \cdot \bfn |_{\Sigma \cap T} \in Q_0(\Sigma\cap T)$.
 Let $\bfv_h=\tilde \bfv_h-\pmb{r}_h$ with $\pmb{r}_h \in [\QkZ]^d$ such that $\pmb{r}_h\cdot\bfn|_{\Sigma \cap T} = \tilde \bfv_h\cdot\bfn|_{\Sigma \cap T}$ for each element $T\in\mesh$.
  Note that $\bfv_h \in \VkZ$, $\dive\bfv_h=\dive\tilde \bfv_h=-q_h$ (so $\Bf(\bfv_h,q_h)=\Bf(\tilde \bfv_h, q_h)$), and 
	\begin{align}
	c(\bfv_h,\chi_h) &= \int_\Sigma (\tilde \bfv_h - \pmb{r}_h)\cdot\bfn \chi_h=0. 
	\end{align}
  Furthermore,
\begin{equation}
        \tn\bfv_h\tn_h \lesssim \|\bfv_h\|_{\bfH^{\dive}(\Omdh)} \lesssim \|\tilde \bfv_h\|_{\bfH^{\dive}(\Omdh)}+ \|\pmb{r}_h\|_{\Omdh}. 
\end{equation}
From \eqref{eq:vh_divstab} we have $\| \tilde \bfv_h \|_{\bfH^{\dive}(\Omdh)} \lesssim  \|q_h\|_{\Omdh}$ and we can choose $\pmb{r}_h\cdot\bft|_T$ for each $T\in \mcT_\Sigma$ so that $\|\bfv_h\|_{\bfH^{\dive}(\Omdh)} \lesssim \|q_h\|_{\Omdh}$. 
\end{proof}

Define the system form $\Af:(\VkZ\times\QkZ\times\QksZ)\times (\VkZ\times\QkZ\times\QksZ)$ and the system norm by 
\begin{align}
	\Af(\bfu_h,p_h,\phi_h; \bfv_h,q_h,\chi_h) &:= \Mf(\bfu_h,\bfv_h) + \Bf(\bfv_h,p_h) + c(\bfv_h,\phi_h) + \Bf(\bfu_h,q_h) + c(\bfu_h,\chi_h) - s_c(\phi_h,\chi_h), \label{eq:systA}\\
	\tn (\bfv_h,q_h,\chi_h) \tn^2_h &:= \tn \bfv_h\tn^2_h + \|q_h\|^2_{\Omdh} + \| \chi_h \|^2_{h}. \label{eq:systnorm}
\end{align}
Notice that $\Af$ is symmetric:
\begin{align}
	\Af(\bfu_h,p_h,\phi_h; \bfv_h,q_h,\chi_h) = \Af(\bfv_h,q_h,\chi_h; \bfu_h,p_h,\phi_h).
\end{align}

\begin{thm}\textbf{(Inf-sup condition)}\label{thm:big_infsup}
	For any $(\bfw_h,\xi_h,\sigma_h)\in \VkZ\times\QkZ\times\QksZ$ there exists $(\bfv_h,q_h,\chi_h)\in \VkZ\times(\oplus_{T\in\mesh} Q_0(T))\times\QksZ$ such that 
	\begin{align}
		\Af(\bfw_h,\xi_h,\sigma_h; \bfv_h,q_h,\chi_h) &= \Mf(\bfw_h,\bfv_h) + \Bf(\bfw_h,q_h) + c(\bfw_h,\chi_h) + \Bf(\bfv_h,\xi_h) + c(\bfv_h,\sigma_h) - s_c(\chi_h,\sigma_h)\nonumber\\
		&\gtrsim \tn (\bfw_h,\xi_h,\sigma_h) \tn_h^2,\\
		\tn (\bfv_h,q_h,\chi_h) \tn_h &\lesssim \tn (\bfw_h,\xi_h,\sigma_h) \tn_h.
	\end{align}
	Moreover, if  $ \int_\Omega \dive\bfw_h= 0$ then $q_h \in \QkZ$. 
\end{thm}  
\begin{proof}
Note that  $\Bf(\bfw_h, -\dive\bfw_h)=\|\dive \bfw_h\|^2_{\Omega_1 \cup \Omega_2}+ s(\dive \bfw_h, \dive \bfw_h) \gtrsim \|\dive \bfw_h\|^2_{h}$, in the last inequality we used Lemma \ref{lem:sp_ineq}. Thus, 
	\begin{align}
		\Af(\bfw_h,\xi_h,\sigma_h;\bfw_h,-\xi_h-\dive\bfw_h,-\sigma_h) &= \Mf(\bfw_h,\bfw_h)+\Bf(\bfw_h,-\xi_h-\dive\bfw_h)+\Bf(\bfw_h,\xi_h)+s_c(\sigma_h,\sigma_h) \nonumber\\
		&=\Mf(\bfw_h,\bfw_h)+\Bf(\bfw_h,-\dive\bfw_h)+s_c(\sigma_h,\sigma_h)  \nonumber \\
		&\gtrsim \tn\bfw_h\tn^2_{\Mf} +  \|\dive \bfw_h\|^2_h +s_c(\sigma_h,\sigma_h)
		= \tn\bfw_h\tn^2_h+s_c(\sigma_h,\sigma_h).\label{eq:u_term}
	\end{align}
	By virtue of Lemma \ref{lem:dim_count} we can pick a function $\bfrho_h\in \VkZ$ satisfying
	\begin{align}
           \Bf(\bfrho_h,\xi_h)\gtrsim \|\xi_h\|^2_{\Omdh}, \  \tn\bfrho_h\tn_h \lesssim \|\xi_h\|_{\Omdh}, \dive\bfrho_h=-\xi_h, \
           c(\bfrho_h,\sigma_h)=h^{-1}\|\sigma_h\|^2_{\Sigma}\in \RR.
       	\end{align}
	By continuity of $\Mf$ and Young's inequality it now follows
	\begin{align}
		\Af(\bfw_h,\xi_h,\sigma_h;\bfrho_h,0,0) &= \Mf(\bfw_h,\bfrho_h) + \Bf(\bfrho_h,\xi_h) + c(\bfrho_h,\sigma_h) \nonumber\\
		&\gtrsim -\tn\bfrho_h\tn_h\tn\bfw_h\tn_h + \|\xi_h\|^2_{\Omdh} + c(\bfrho_h,\sigma_h) \nonumber\\
		&\gtrsim-\|\xi_h\|_{\Omdh}\tn\bfw_h\tn_h + \|\xi_h\|^2_{\Omdh} + c(\bfrho_h,\sigma_h) \nonumber\\
		&\gtrsim -1/2\tn\bfw_h\tn^2_h + 1/2\|\xi_h\|^2_{\Omdh} + h^{-1}\|\sigma_h\|^2_{\Sigma}. \label{eq:lam_term}
	\end{align}
	We pick $(\bfv_h,q_h,\chi_h)=(\bfw_h+\bfrho_h,-\xi_h-\dive\bfw_h,-\sigma_h)$ and use Lemma \ref{lem:lagrange}, hence we have 
	\begin{align}
		\Af(\bfw_h,\xi_h,\sigma_h; \bfv_h,q_h,\chi_h)
          &\gtrsim 1/2 \tn\bfw_h\tn^2_h + 1/2\|\xi_h\|^2_{\Omdh} +h^{-1}\|\sigma_h\|^2_{\Sigma}+s_c(\sigma_h,\sigma_h)
           \gtrsim \tn (\bfw_h,\xi_h,\sigma_h) \tn^2_h.
	\end{align}
	By design, the following inequality holds
	\begin{align}
		\tn \bfrho_h\tn_h&\lesssim \|\xi_h\|_{\Omdh}, 
	\end{align}
	and consequently
	\begin{align}
		\tn (\bfv_h,q_h,\chi_h) \tn_h \lesssim \tn \bfw_h+\bfrho_h\tn_h+ \|\dive\bfw_h+\xi_h\|_{\Omdh} + \|\sigma_h\|_{h}
		\lesssim \tn (\bfw_h,\xi_h,\sigma_h) \tn_h.
	\end{align}
	Finally, since $q_h=-\xi_h-\dive\bfw_h$ and $\xi_h\in \QkZ$ we have $q_h\in \QkZ$ if $ \int_\Omega \dive\bfw_h= 0$.
\end{proof}

\subsection{Interpolation estimates and consistency}
To define the interpolant we need extensions of functions in $\Omega$ to $\Omdh$. We will use the Sobolev-Stein extension operators of \cite[(3.16) and Corollary 4.1]{Hiptmair2012Universal}:
\begin{align}
	\bfE: \bfH^{1,\dive}(\Omega)\to \bfH^{1,\dive}(\RR^d),\
        E: H^{1}(\Omega)\to H^{1}(\RR^d), 
\end{align}
with $k\geq 0$, which for $\bfv\in\bfH^{1,\dive}(\Omega)$ and $q\in H^{1}(\Omega)$ satisfy
\begin{align}
\|\bfE\bfv\|_{\bfH^{1,\dive}(\RR^d)} &\lesssim \|\bfv\|_{\bfH^{1,\dive}(\Omega)},
\quad (\bfE \bfv)|_{\Omega}=\bfv \text{ a.e.},
\label{eq:sobo_velbound}\\
\|E  q\|_{H^{1}(\RR^d)} &\lesssim \|q\|_{H^{1}(\Omega)}, \quad
 (E q)|_{\Omega}=q \text{ a.e.}, \label{eq:sobo_presbound}\\
\dive\circ \bfE &= E \circ \dive.\label{eq:commdiagE}
\end{align} 

The last property \eqref{eq:commdiagE} allow us to define the following interpolation and projection operators
\begin{align}\label{eq:defintpv}
	\interpv &:  \bfH^{1,\dive} (\Omega) \ni \bfv \mapsto \interpv (\bfE\bfv|_{\Omdh}) \in \VkZ, \\
	\label{eq:defintpp}
	\projp &:  H^{1}(\Omega) \ni q \mapsto \projp (Eq|_{\Omdh})\in \QkZ,
\end{align}
which satisfy $\dive \interpv \bfv = \projp \dive \bfv$ for all $\bfv \in \bfH^{1,\dive}(\Omega)$.

For the continuous Lagrange multiplier variable representing the velocity normal component on $\Sigma$ we introduce the extension operator 
\begin{align}
  \Es : H^{2}(\Sigma) \to H^{2}(\RR^d)
  %H^{1/2}(\Sigma) \to H^{1}(\RR^d)
\end{align}
Then we can also define the following $L^2$-projection operator 
 \begin{align}
 	\projp^{\Sigma} &:  H^{2}(\Sigma) \ni \phi \mapsto \projp^{\Sigma} (\Es \phi|_{\Omsig})\in \QksZ,
 \end{align}
 and we note that from \cite[Lemma 3.1]{Reu15} we also have 
 \begin{align}
 	\| \Es \phi \|_{H^s(U_\delta(\Sigma))} \lesssim \delta^{1/2}\| \phi \|_{H^s(\Sigma)}, \quad s\geq 0.\label{eq:stabEs}
 \end{align}

\begin{lemma}\textbf{(Interpolation/projection estimates)}\label{lem:interp}
	Let $\bfu\in  \bfH^{1,\dive}(\Omega)$, $p\in  H^{1}(\Omega)$, and $\phi \in H^2(\Sigma)$, then
        \begin{align}
	 \tn \bfE \bfu-\interpv \bfu \tn_h &\lesssim h \left(\|\bfu\|_{\bfH^{1}(\Omega)} + \|\dive\bfu\|_{H^{1}(\Omega)} \right),\label{eq:interpresu} \\
         \| Ep-\projp p \|_{\Omdh} &\lesssim h  \|p\|_{H^{1}(\Omega)},\label{eq:interpresp}\\
\|\Es \phi-\projp \phi \|_h &\lesssim h^{3/2} \| \phi\|_{H^2(\Sigma)}
  \label{eq:interpreslag}                             
        \end{align}
\end{lemma}
\begin{proof}
Let $\bfrho= \bfE \bfu-\interpv \bfu$. Recall that 
	\begin{align*}
		\tn \bfrho \tn_h^2 &= 
		\|\permi^{1/2}\bfrho \|^2_{\Omega_1\cup\Omega_2}  + \|\dive\bfrho \|^2_{\Omdh} +
		s(\bfrho, \bfrho)	+
		\|h^{1/2}\bfrho\cdot\bfn\|^2_{\Sigma}.
	\end{align*}
Using local interpolation estimates from Lemma \ref{lem:loc_interp} and the properties of the extension operators, \eqref{eq:sobo_velbound}-\eqref{eq:commdiagE}, we have
\begin{align}
	\|\permi^{1/2}\bfrho\|^2_{\Omega}& \lesssim \sum_{T\in \mesh}\|\bfrho\|^2_{T} \lesssim h^{2} \sum_{T\in \mesh} |\bfE \bfu |^2_{1,T} \lesssim h^{2} \| \bfu \|^2_{\bfH^{1}(\Omega)}, \\
	\|\dive \bfrho \|^2_{\Omdh} &= \sum_{T\in \mesh} \|\dive \bfrho \|^2_{T} \lesssim h^{2} \sum_{T\in \mesh} |\dive \bfE\bfu |^2_{1,T} 
	=h^{2} \sum_{T\in \mesh} |E \dive \bfu|_{1,T}^2 \\
	&\lesssim h^{2} \| \dive\bfu \|^2_{H^{1}(\Omega)}. 
\end{align}

Next we look at the stabilization term. We use an element-wise trace inequality, followed by local interpolation estimates (Lemma \ref{lem:loc_interp}), and the stability of the extension operator to get
\begin{align}
	s(\bfrho, \bfrho)&=  \sum_{F \in \mcFs}
	\tau (h \|\jump{\bfrho }\|^2_{F} + h^{3}  \|\jump{D_{\bfn_F} \bfrho }\|^2_{F} ) \lesssim   \sum_{T \in \mcT_{\Sigma}} (\|\bfrho \|^2_{T} + h^{2}  |\bfrho |^2_{1,T} ) \nonumber \\
	&\lesssim h^{2}  \sum_{T\in \mesh} | \bfE \bfu |^2_{1,T} \lesssim h^{2}  \| \bfu \|^2_{\bfH^{1}(\Omega)}. \label{eq:interpboundII}
\end{align}
For the boundary term we again use the trace inequality \label{eq:trace_ineq}, followed by local interpolation estimates, and the stability of the extension operator:
\begin{align}
	\|h^{1/2}\bfrho \cdot \pmb{n}\|^2_{\Sigma} & 
	= \sum_{T\in\mesh}\|h^{1/2}\bfrho \cdot \pmb{n}\|^2_{T\cap \Sigma}  \lesssim \sum_{T\in\mesh} \|\bfrho\|_T^2+h^2|\bfrho|_{1,T}^2\lesssim h^2\sum_{T\in\mesh}| \bfE \bfu |^2_{1,T}
	\lesssim h^2\|\bfu\|_{\bfH^{1}(\Omega)}^2.
\end{align} 
Summing all terms together we get the desired estimate \eqref{eq:interpresu}. Since $\|\cdot \|_{\Omdh}$ is just the broken $L^2$-norm on the active meshes, the estimate \eqref{eq:interpresp} follows from standard estimates and using the stability \eqref{eq:sobo_presbound} of the extension operator $E$. 

For the last estimate we use the trace inequality, standard estimates for the $L^2$-projection, and the stability of the extension operator $\Es$, \eqref{eq:stabEs} with $\delta \sim h$ to get 
\begin{align}
  \|\Es \phi-\projp \phi \|_h&=
  h^{-1/2} \|\Es \phi-\projp \phi \|_\Sigma + s_c(\Es \phi-\projps \phi,\Es \phi-\projps \phi)^{1/2} \\
  &\lesssim
          \sum_{T\in\Omega_\Sigma}h^{-1} \|\Es \phi-\projp \phi \|_T+\|(\Es \phi-\projp \phi) \|_{1,T}
          \lesssim
          \sum_{T\in\Omega_\Sigma} h \|\Es \phi \|_{2,T} \lesssim h^{3/2} \|\phi \|_{H^2(\Sigma)}.\nonumber 
	\end{align}
\end{proof}

\begin{lemma}(\textbf{Consistency})\label{lem:consistency}
	Let $(\bfu,p,\phi) \in (\bfH^{\dive}_{u_B}(\Omega) \cap  \bfH^{1,\dive}(\Omega))\times H^{1}(\Omega)/\RR \times H^{2}(\Sigma)$ be the solution to \eqref{eq:weakdarcy}.
	Let $(\bfu_h,p_h,\phi_h)\in\VkZ\times\QkZ \times \QksZ$ be the solution to \eqref{eq:discretedarcy2}. Then,
	\begin{align}
		\Mf(\bfu_h-\bfE \bfu,\bfv_h) + \Bf(\bfv_h, p_h-E p) + c(\bfv_h,\phi_h-\Es \phi) &=-s(\bfE \bfu,\bfv_h),
			\quad &\forall \bfv_h \in \VkZ,\\
		\Bf(\bfu_h-\bfE \bfu,q_h) &= 0, \quad &\forall q_h \in \QkZ,\\
		c(\bfu_h-\bfE\bfu,\chi_h)-s_c(\phi_h-\Es \phi,\chi_h) &= 0, \quad &\forall \chi_h\in\QksZ,
	\end{align}
	with
	\begin{align}\label{eq:suforEu}
	s(\bfE \bfu,\bfv_h)=   \sum_{F \in \mathcal{F}_{h,i}}
	\tau h^{3}  (\jump{D_{\bfn_F} \bfE \bfu }, \jump{D_{\bfn_F} \bfv_{h}})_{F}.
	\end{align}
\end{lemma}
\begin{proof} 
  Note that since $ Ep \in H^{1}(\RR^d)$ and $\bfE \bfu \in \bfH^{1,\dive}(\RR^d)$ we have $\jump{  Ep}=0$, $\jump{\dive (\bfE \bfu)}=0$, and $\jump{ \bfE \bfu}=0$. 
  Hence $s_b(\bfv_h,Ep)=s_b(\bfE \bfu, q_h)=0$ and consequently $\Bf(\bfE \bfu, q_h)=b(\bfu, q_h)$ and $\Bf(\bfv_h, E p)=b(\bfv_h, p)$. 
  However, $s(\bfE \bfu,\bfv_h)\neq 0$ and instead equation \eqref{eq:suforEu} holds. 
  Moreover $s_c(\Es \phi,\chi_h)=0$ since $\Es \phi\in H^{2}(\RR^d)$. 
  The continuous and discrete variables are solutions to their respective equations, and by using this we get 
	\begin{align}
	\Mf(\bfu_h-\bfE \bfu,\bfv_h) + \Bf(\bfv_h, p_h-E p)+c(\bfv_h,\phi_h-\Es \phi) &= -\sum_{F \in \mathcal{F}_{h,i}}
                                                                                        \tau h^{3} ([D_{\bfn_F} \bfE \bfu ], [D_{\bfn_F} \bfv_{h}])_{F},  	\quad &\forall \bfv_h \in \VkZ,\nonumber \\
          \Bf(\bfu_h-\bfE \bfu,q_h) &= 0,   \quad &\forall q_h \in \QkZ,\nonumber \\
	c(\bfu_h-\bfE\bfu,\chi_h)-s_c(\phi_h-\Es \phi,\chi_h) &= 0,  \quad &\forall \chi_h\in\QksZ,\nonumber
	\end{align}
	and the result follows. 
\end{proof}

\subsection{A priori estimate}\label{sec:apriori}
We now prove theoretical estimates for the convergence order of the proposed method. We must assume more regularity than $L^2(\Omega)$ for the divergence of the solution in order to get optimal a priori estimates.

\begin{thm}\textbf{(A priori error estimate)}\label{thm:apriori}
	Let $(\bfu,p,\phi) \in (\bfV \cap  \bfH^{1,\dive}(\Omega))\times H^{1}(\Omega) \times H^{2}(\Sigma)$ be the solution to \eqref{eq:weakdarcy}. Let $(\bfu_h,p_h,\phi_h)\in\VkZ\times\QkZ \times \QksZ$ be the solution to \eqref{eq:discretedarcy2}. Then,
	\begin{align}
		\tn \bfE \bfu-\bfu_h \tn_h
		+	\| Ep-p_h \|_{\Omdh} +	\| \Es \phi-\phi_h \|_{h}
		&\lesssim h \left(\|\bfu\|_{\bfH^{1}(\Omega)} + \|\dive\bfu\|_{H^{1}(\Omega)}+\|p\|_{H^{1}(\Omega)}+h^{1/2} \| \phi\|_{H^2(\Sigma)} \right).
	\end{align}
\end{thm}
\begin{proof}
	Adding and subtracting the interpolant and using the triangle inequality we have,
	\begin{align}
		\tn \bfE\bfu-\bfu_h \tn_h &\leq  \tn \bfE\bfu-\interpv\bfu \tn_h + \tn \bfu_h-\interpv\bfu \tn_h, \\
          \|Ep-p_h\|_{\Omdh}&\leq \|Ep-\projp p\|_{\Omdh} + \|p_h-\projp p\|_{\Omdh},\\
          \|\Es \phi-\phi_h\|_{h}&\leq \|\Es \phi-\projps \phi\|_{h} + \|\phi_h-\projps \phi\|_{h}.
	\end{align}
	We now seek to estimate  $\tn (\bfu_h-\interpv\bfu,p_h-\projp p,\phi_h-\projps \phi) \tn_h =\tn \bfu_h-\interpv\bfu \tn_h +\|p_h-\projp p\|_{\Omdh}+\|\phi_h-\projps \phi\|_{h}$. 
        Applying Theorem \ref{thm:big_infsup} to $(\bfw_h, \xi_h, \sigma_h)=(\bfu_h-\interpv\bfu, p_h-\projp p,\phi_h-\projps \phi)$, we have that there exist $(\bfv_h,q_h,\chi_h)\in \VkZ\times\QkZ\times\QksZ$ such that 
	\begin{align}
		\tn (\bfw_h,\xi_h,\sigma_h) \tn_h^2 &\lesssim  \Af(\bfw_h,\xi_h,\sigma_h;\bfv_h,q_h,\chi_h), \label{eq:DinfsupEst}\\ 
		\tn (\bfv_h,q_h,\chi_h) \tn_h &\lesssim \tn (\bfw_h,\xi_h,\sigma_h) \tn_h.\label{eq:boundnorm}
	\end{align}
	Consistency, Lemma \ref{lem:consistency}, yields 
	\begin{align}\label{eq:afconst}
		\Af(\bfw_h,\xi_h,\sigma_h;\bfv_h,q_h,\chi_h)
		&= \Mf(\bfu_h-\interpv\bfu, \bfv_h) + \Bf(\bfv_h,p_h-\projp p) + c(\bfv_h,\phi_h-\projps \phi) \\
		&\quad + \Bf( \bfu_h-\interpv\bfu,q_h) + c( \bfu_h-\interpv\bfu,\chi_h) - s_c(\phi_h-\projps \phi,\chi_h) \nonumber\\
		&= \Mf(\bfE\bfu-\interpv\bfu, \bfv_h) + \Bf(\bfv_h,Ep-\projp p) +c(\bfv_h, \Es \phi-\projps \phi)- s(\bfE\bfu,\bfv_h) \nonumber\\
		&\quad + \Bf(\bfE\bfu-\interpv\bfu,q_h) + c(\bfE\bfu-\interpv\bfu,\chi_h) -s_c(\Es \phi -\projps \phi,\chi_h). \nonumber
	\end{align}
Using the Cauchy-Schwartz inequality we have 
	\begin{align}\label{eq:contcs}
	  c(\bfv_h,\Es \phi-\projps \phi )-s_c(\Es \phi-\projps \phi, \chi_h )  
          \lesssim &\|h^{1/2}\bfv_h\cdot\bfn\|_{\Sigma}\|h^{-1/2}(\Es \phi-\projps \phi) \|_{\Sigma} \\
                   &+ s_c(\Es \phi-\projps \phi,\Es \phi-\projps \phi)^{1/2}s_c(\chi_h,\chi_h)^{1/2}.\nonumber
	\end{align}
     Using that $s(\bfE\bfu,\bfv_h)$ is defined as in equation \eqref{eq:suforEu}, adding and subtracting the interpolant $\pi_h\bfu$, applying the Cauchy-Schwartz inequality together with continuity of the bilinear forms $\Mf$, $\Bf$, $c$ (Lemma~\ref{lem:continuity}), and \eqref{eq:contcs} we get from equation \eqref{eq:afconst} that
     \begin{align}
		\Af(\bfw_h,\xi_h,\sigma_h;\bfv_h,q_h,\chi_h)
		&\lesssim 
                \left(\tn \bfE\bfu-\interpv\bfu \tn_h+\|Ep-\projp p\|_{\Omdh} + \left(\sum_{F\in\mcF_h} h^{3}\|D_{\bfn_F}\interpv\bfu\|_F^2\right)^{1/2}
                +\|\Es \phi-\projp \phi \|_h
                %h^{-1/2}\|\Es \phi-\projp \phi \|_\chi + s_c(\Es \phi-\projps \phi,\Es \phi-\projps \phi)^{1/2}
           \right)  \nonumber \\
       & \quad \cdot \tn (\bfv_h,q_h,\chi_h) \tn_h.
	\end{align}
   
 Using \eqref{eq:DinfsupEst} followed by \eqref{eq:boundnorm} we end up with the estimate
	\begin{align}
		\tn &(\bfu_h-\interpv\bfu, p_h-\projp p,\phi_h-\projps \phi) \tn_h \\
		&\lesssim  \tn \bfE\bfu-\interpv\bfu \tn_h+\|Ep-\projp p\|_{\Omdh}  +\left( \sum_{F \in \mcF_h} h^{3} \| D_{\bfn_F} \interpv \bfu\|_F^2 \right)^{1/2} +\|\Es \phi-\projp \phi \|_h. \nonumber 
	\end{align} 
	Using a standard element-wise trace inequality, an inverse estimate, local interpolation estimates (Lemma \ref{lem:loc_interp}), and the stability of the extension operator we have  
	\begin{align} \label{eq:consterr}
	   \sum_{F \in \mcF_h} h^{3} \| D_{\bfn_F} \interpv \bfu\|_F^2 
	   &\lesssim \sum_{T \in \mesh} h^{2} \| \nabla \interpv \bfu\|_T^2   \\
	  &\lesssim \sum_{T \in \mesh} h^{2} \left( \| \nabla (\interpv \bfu-\bfE\bfu)\|_T^2+  \| \nabla \bfE\bfu \|_T^2 \right)  \lesssim
	  h^{2} \|\bfu\|_{\bfH^{1}(\Omega)}^2. \nonumber
	\end{align}
	Finally, using the interpolation and projection error estimates in Lemma \ref{lem:interp} we get:
	\begin{align*}
          \tn (\bfE\bfu-\bfu_h, Ep-p_h,\Es\phi_h-\phi_h) \tn_h 
             &\lesssim
               \tn \bfE\bfu-\interpv\bfu \tn_h +\|Ep-\projp p\|_{\Omdh} \\
               &+\left( \sum_{F \in \mcF_h} h^{3} \| D_{\bfn_F} \interpv \bfu\|_F^2 \right)^{1/2} 
               +\|\Es \phi-\projp \phi \|_h \\
		&\lesssim h\left(\|\bfu\|_{\bfH^{1}(\Omega)} + \|\dive\bfu\|_{H^{1}(\Omega)}+\|p\|_{H^{1}(\Omega)} + h^{1/2} \| \phi\|_{H^2(\Sigma)}\right).
	\end{align*}

\end{proof}

\section{Condition number estimate}\label{sec:condnbr}
Based on the results of the previous section we prove that the spectral condition number of the resulting linear system scales as $h^{-2}$ independent of how the background mesh cuts the boundary. We will need the $L^2$-norm of $(\bfv_h,q_h,\chi_h)\in\VkZ\times\QkZ\times\QksZ$;
\begin{align}\label{eq:L2_system_norm}
	\|(\bfv_h,q_h,\chi_h)\|^2_{\Omdh} := \|\bfv_h\|^2_{\Omdh} + \|q_h\|^2_{\Omdh} + \|\chi_h\|^2_{\Omdh}\ = \|\bfv_h\|^2_{\Omdh} + \|q_h\|^2_{\Omdh} + \|\chi_h\|^2_{\Omsig}.
\end{align}

Let $N:=\dim(\VkZ\times\QkZ\times\QksZ)$ and denote by $\vec{v} \in \RR^{N}$ the unique vector containing expansion coefficients of $(\bfv_h,q_h,\chi_h)$ in the basis of the spaces $\VkZ\times\QkZ\times\QksZ$. Similarly we denote by $\vec{w}$ the unique coefficients of $(\bfw_h,r_h,\sigma_h)\in \VkZ\times\QkZ\times\QksZ$. 
Let $\|\vec{v}\|_2=\sqrt{\vec{v}\cdot\vec{v}}$ be the Euclidean $2$-norm of $\vec{v}$. Similarly let $A\in\RR^{N\times N}$ be the matrix associated with $\Af$, defined by 
\begin{equation}
A\vec{v}\cdot\vec{w} = \Af((\bfv_h,q_h,\chi_h),(\bfw_h,r_h,\sigma_h)).  
\end{equation}
Since the matrix $A$ is symmetric, the spectral condition number of the matrix is defined by 
\begin{align}
  \kappa (A) =\|A\|_2\|A^{-1}\|_2 = \frac{\lambda_{\textrm{max}}}{\lambda_{\textrm{min}}},
%	\|A\|_2:=\max_{\vec{v}\in\RR^N\setminus\{0\}} \frac{\|A\vec{v}\|_2}{\|\vec{v}\|_2}.
\end{align}
where $\lambda_{\textrm{max}}$ and $\lambda_{\textrm{min}}$ denote the largest and the smallest eigenvalues of $A$. 
 
The following result holds by our assumptions on the mesh \cite[Lemma A.1]{Ern2006Evaluation}.
\begin{lemma}\label{lem:condlem1}
	For every $(\bfv_h,q_h,\chi_h)\in\VkZ\times\QkZ\times\QksZ$,
	\begin{align}
		h^{d/2}\|\vec{v}\|_2 \lesssim \|(\bfv_h,q_h,\chi_h)\|_{\Omdh} \lesssim h^{d/2}\|\vec{v}\|_2.
	\end{align}
\end{lemma}

\begin{lemma}\label{lem:condlem2}
	For every $(\bfv_h,q_h,\chi_h)\in\VkZ\times\QkZ\times\QksZ$,
	\begin{align}
		\|(\bfv_h,q_h,\chi_h)\|_{\Omdh} \lesssim \tn (\bfv_h,q_h,\chi_h)\tn_h \lesssim h^{-1}\|(\bfv_h,q_h,\chi_h)\|_{\Omdh}.
	\end{align}
\end{lemma}
\begin{proof}
	Recall the definition \eqref{eq:systnorm}:
	\begin{align}
		\tn (\bfv_h,q_h,\chi_h) \tn^2_h = \tn \bfv_h\tn^2_h + \|q_h\|^2_{\Omdh} + \| \chi_h \|^2_{h}. \nonumber
	\end{align}
        By Lemma~\ref{lem:lagrange} we have  
	\begin{align}
		\|\chi_h\|_{\Omsig} \lesssim \|h^{-1}\chi_h\|_{\Omsig} \lesssim \|\chi_h\|_{h} \lesssim \| h^{-1}\chi_h\|_{h}.
	\end{align} 
	We are left to prove that 
	\begin{align}
		\|\bfv_h\|_{\Omdh} \lesssim \tn \bfv_h\tn_h \lesssim h^{-1}\|\bfv_h\|_{\Omdh}.
	\end{align}
        Using \eqref{eq:sa_stabbound} and the definition of $\tn \bfv_h\tn_h$ we have
\begin{equation}
  \|\bfv_h\|_{\Omdh} \lesssim \tn \bfv_h \tn_h   %=\|\permi^{1/2}\bfv_h\|^2_{\Omega} + s(\bfv_h,\bfv_h) + \| \dive \bfv_h\|_{\Omdh}^2 + \|h^{1/2}\bfv_h\cdot \pmb{n}\|^2_{\Sigma}
  \lesssim  \|\bfv_h\|_{\Omdh}  + \| \dive \bfv_h\|_{\Omdh}^2 + \|h^{1/2}\bfv_h\cdot \pmb{n}\|^2_{\Sigma}.
        \end{equation}
 Using an inverse inequality on the divergence term and the element-wise trace inequality \eqref{eq:discrete_trace_ineq}  we have       
	\begin{align}
		\|\dive \bfv_h\|_{\Omdh} &\lesssim h^{-1} \|\bfv_h\|_{\Omdh}, \\
		\|h^{1/2}\bfv_h\cdot\bfn\|_{\Sigma} &\lesssim h^{-1/2} \|\bfv_h\|_{\Omdh},
	\end{align}
        and the desired estimate follows. 
\end{proof}
\begin{thm}\textbf{(Condition number scaling)}
  The following bound holds for the spectral condition number
  %of $A$
	\begin{align}
	\kappa (A) =\frac{\lambda_{\textrm{max}}}{\lambda_{\textrm{min}}} \lesssim h^{-2}.
	\end{align}
\end{thm}
\begin{proof}
  We have that
  \begin{equation}
    \lambda_{\textrm{max}} = \max_{\vec{v}\in\RR^N\setminus\{0\}} \frac{A\vec{v}\cdot\vec{v}}{\|\vec{v}\|_2}, \quad
    \lambda_{\textrm{min}} = \min_{\vec{v}\in\RR^N\setminus\{0\}} \frac{A\vec{v}\cdot\vec{v}}{\|\vec{v}\|_2}.
  \end{equation}  
	By Lemma \ref{lem:continuity} (continuity),  Lemma \ref{lem:condlem2}, and Lemma \ref{lem:condlem1}, we have
	\begin{align}
		A\vec{v}\cdot\vec{w} &= \Af((\bfv_h,q_h,\chi_h),(\bfw_h,r_h,\sigma_h)) \lesssim \tn (\bfv_h,q_h,\chi_h)\tn_h \tn (\bfw_h,r_h,\sigma_h)\tn_h \nonumber\\
		&\lesssim h^{-2}\|(\bfv_h,q_h,\chi_h)\|_{\Omdh}\|(\bfw_h,r_h,\sigma_h)\|_{\Omdh} 
		\lesssim h^{d-2}\|\vec{v}\|_2\|\vec{w}\|_2,
	\end{align}
	whereby $ \lambda_{\textrm{max}}\lesssim h^{d-2}.$ By Theorem \ref{thm:big_infsup} we also have the existence of a $\vec{w}$ associated to any $\vec{v}$ such that 
	\begin{align}
		A\vec{v}\cdot\vec{w} &= \Af((\bfv_h,q_h,\chi_h),(\bfw_h,r_h,\sigma_h)) \gtrsim \tn (\bfv_h,q_h,\chi_h)\tn_h \tn (\bfw_h,r_h,\sigma_h)\tn_h \nonumber\\ 
		&\gtrsim \|(\bfv_h,q_h,\chi_h)\|_{\Omdh}\|(\bfw_h,r_h,\sigma_h)\|_{\Omdh} 
		\gtrsim h^{d}\|\vec{v}\|_2\|\vec{w}\|_2, \label{eq:blessed_ineq}
	\end{align}
	where we again used Lemma \ref{lem:condlem1} and \ref{lem:condlem2}. Hence $\lambda_{\textrm{min}}\gtrsim h^{d}$. Combining the bounds we obtain the desired estimate 
	\begin{equation}
          \frac{\lambda_{\textrm{max}}}{\lambda_{\textrm{min}}}  \lesssim \frac{h^{d-2}}{ h^{d}} \lesssim h^{-2}.
	\end{equation}
\end{proof}

\section{Numerical experiments}\label{sec:numex}
We consider three examples which test the method in different ways. The examples are implemented in the open source CutFEM library written in C++ and based on FreeFEM, and the code is available at \cite{repo}. 
As a short-hand we will in this section write $Q_k$ for $\Qk$ and $\RT_k$ for $\Vk$. 
We consider two element triples, namely $\RT_0\times Q_0\times Q_1^{\Sigma}$ and $\RT_1\times Q_1\times Q_2^{\Sigma}$. In order to avoid computations of derivatives we use the extension based stabilization \eqref{eq:patch_1} and \eqref{eq:patch_2} of $s$ and $s_b$, respectively. For all examples we choose the stabilization parameters to be equal to one, i.e.,  $\tau=\tau_b=\tau_c=1$.
The focus will be on investigating the convergence of the velocity, its divergence and the pressure. The $1$-norm estimate (MATLAB's function condest) condition number is also evaluated for each example. We solve each linear system with the direct solver UMFPACK.

\begin{itemize}
\item In Example 1 we have an exact solution where the divergence is not in the approximation space for the pressure $Q_k$, for any $k\geq 0$. The boundary $\partial \Omega_h$ is a piecewise linear approximation of a circle in an unfitted mesh on which pure essential conditions are imposed using the proposed stabilized Lagrange multiplier method. We compare different stabilization terms for the multiplier. 
	\item In Example 2 we have an exact solution with divergence in the space $Q_1$. We use a rectangular domain and compare the proposed cut finite element method with a standard fitted finite element method where boundary conditions are imposed strongly.
	\item In Example 3 we consider an interface problem wherein both the interface and the domain boundary are unfitted. This example exhibits mixed boundary conditions.
\end{itemize}

\subsection{Example 1 (divergence outside pressure space)}
Let $\Omega$ be the disk with boundary $\Sigma=\partial \Omega$ the circle with radius $0.45$. We let $\Omega_0=[0,1]^2$. We take $\permi = 1$ and the exact solution to problem \eqref{eqs:strongDarcy} is  
\[ \pmb{u} = (2\pi\cos(2\pi x)\cos(2\pi y), -2\pi\sin(2\pi x)\sin(2\pi y))^T,\quad p=-\sin(2\pi x)\cos(2\pi y). \] 
The Dirichlet boundary condition is enforced weakly everywhere on a piecewise linear approximation of $\partial \Omega$. See Figure \ref{fig:example 2 solution} for a heatmap of the magnitude of the approximated velocity field obtained with the element triple $\RT_1\times Q_1\times Q_2^{\Sigma}$. 
We study the stabilization term $s_c$~\eqref{eq:full stab-c} and the alternative stabilization $\tilde s_c$~\eqref{eq:alt stab-c} and $\hat s_c$~\eqref{eq:alt stab-c 2}. 

\begin{figure}[ht!]
	\centering
	\begin{subfigure}[b]{0.45 \textwidth}  	
		\centering	
		\includegraphics[scale=0.3]{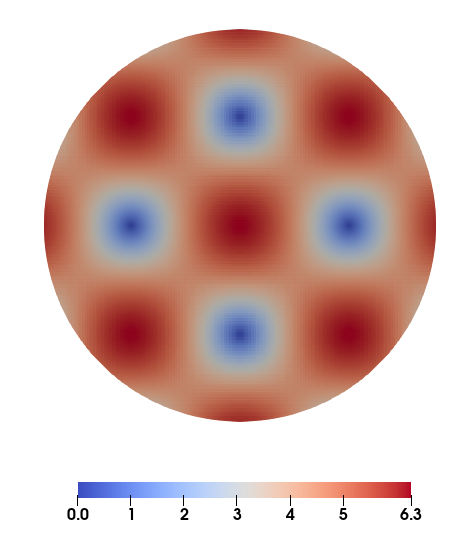}  
	\end{subfigure}
	\caption{Example 1: Magnitude of the approximated velocity field with the element triple $\RT_1\times Q_1\times Q_2^{\Sigma}$. 
	}
	\label{fig:example 2 solution}
\end{figure}

In this example the exact divergence is not in the approximation space of the pressure for any of the considered element triples.   
We show the $L^2$- and the $L{^\infty}$-error of the divergence in Figure \ref{fig:diverror example 2}. 
For the proposed method (regardless of choice of stabilization $s_c$ versus $\hat s_c$), the $L^2$-error and $L^{\infty}$-error of the divergence converges with optimal order. Notice that this is not the case in general for a method which does not satisfy Theorem \ref{thm:divpres}, we showcase this in \cite{FraHaNilZa22}.

\begin{figure}[ht]
	\centering
	\includegraphics[scale=0.40]{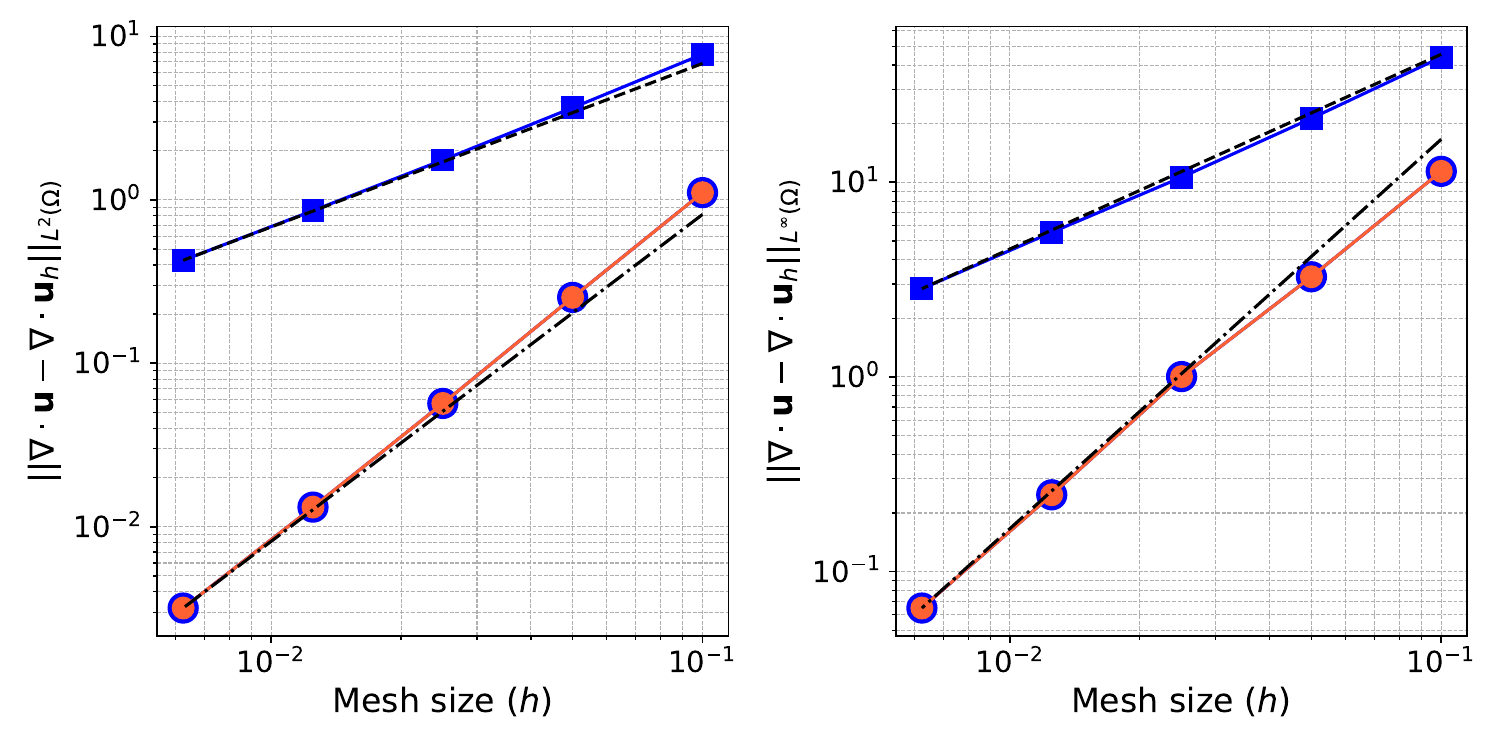}\\
	\includegraphics[scale=0.6]{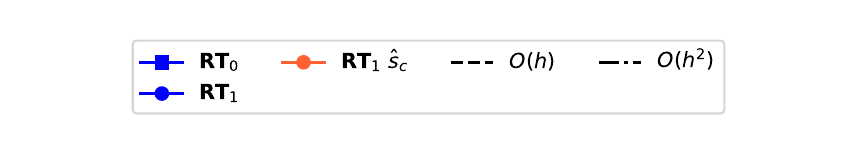}
	\caption{Example 1: The divergence error versus mesh size $h$, using element triples $\RT_0\times Q_0 \times Q^{\Sigma}_1$ and $\RT_1\times Q_1\times Q^{\Sigma}_2$. The different stabilization terms for the Lagrange multiplier are compared.  
		Left: The $L^2$-error of the divergence. Right: The pointwise error of the divergence.}
	\label{fig:diverror example 2}      
\end{figure}

\begin{figure}[ht]
	\centering 
	\includegraphics[scale=0.40]{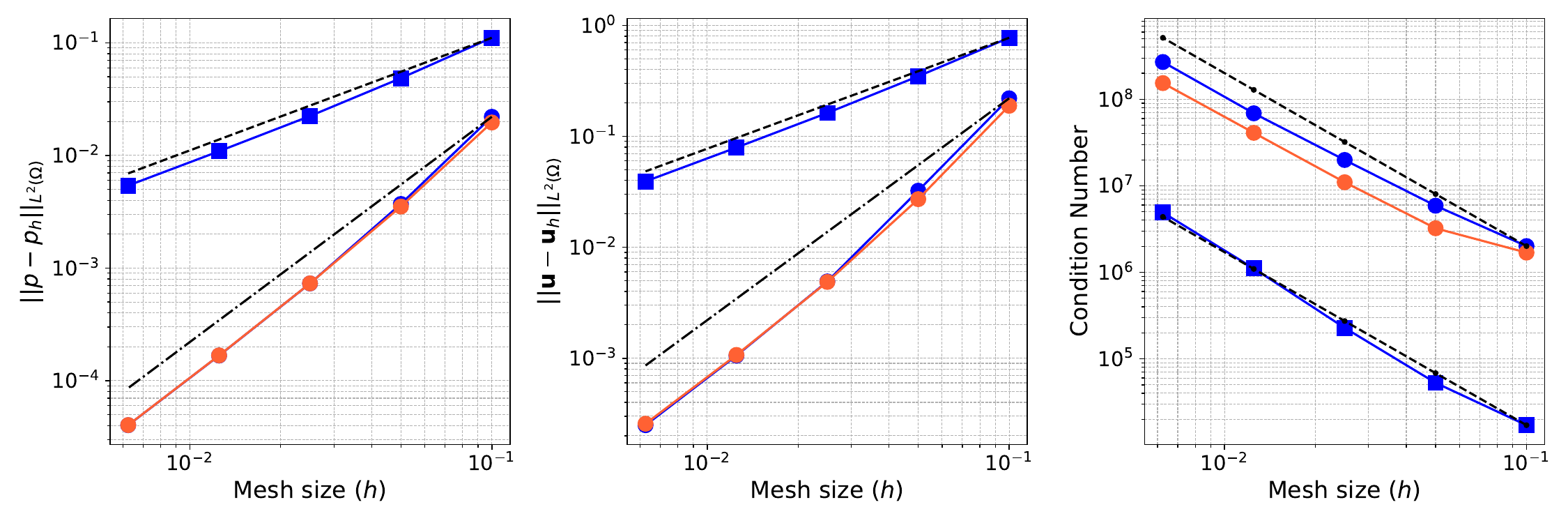} \\
	%\vspace{-0.5cm}
	\includegraphics[scale=0.6]{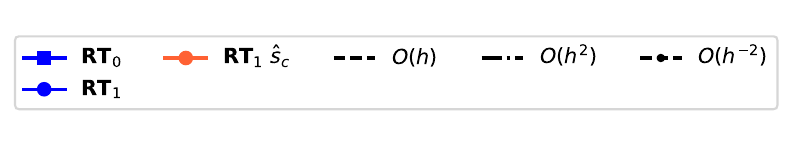}
	\caption{ 
		Example 1: Convergence and condition numbers using element triples $\RT_0\times Q_0 \times Q^{\Sigma}_1$ and $\RT_1\times Q_1\times Q^{\Sigma}_2$ against mesh size $h$. The different stabilization terms for the Lagrange multiplier are compared.  
		Left: The $L^2$-error of the pressure. Middle:  The $L^2$-error of the velocity field. Right: The $1$-norm condition number.
		\label{fig:plot example2 convergence} }                     
\end{figure}

The convergence of the approximated velocity and pressure in the $L^2$-norm and the condition number of the associated matrix are shown in Figure \ref{fig:plot example2 convergence}. Regardless of stabilization $s_c$ versus $\hat s_c$, the method exhibits optimal convergence order for the proposed element triples, and an optimal $O(h^{-2})$ scaling of the condition number.
However, for the stabilization $\hat s_c$ the condition number has a smaller magnitude. We have not optimized the constants in the stabilization terms. In general the results are very similar for these two stabilization terms.

\subsubsection{Comparison of the alternative stabilization terms}\label{sec:alt stab}
Let us now compare the performance of $\tilde s_c$~\eqref{eq:alt stab-c} with $\hat s_c$~\eqref{eq:alt stab-c 2}, which is our proposed modification of $\tilde s_c$. We use $k=1$. For the Lagrange multiplier space we use either $Q_2^{\Sigma}$ or $Q_1^{\Sigma}$. The results are shown in Figure \ref{fig:plot example2 3alts}. We see that the condition number of the variant using $\tilde s_c$ is less stable compared to the variant $\hat s_c$, especially when using $Q_2^{\Sigma}$. The convergence order of $\tilde s_c$ drops slightly in the last data points when using $Q_1^{\Sigma}$, this is likely just the error jumping up to a new asymptotic constant. However, using $\hat s_c$ with $Q_1^{\Sigma}$ the convergence is stable and of order two.
\begin{figure}[ht]
	\centering 
	\includegraphics[scale=0.40]{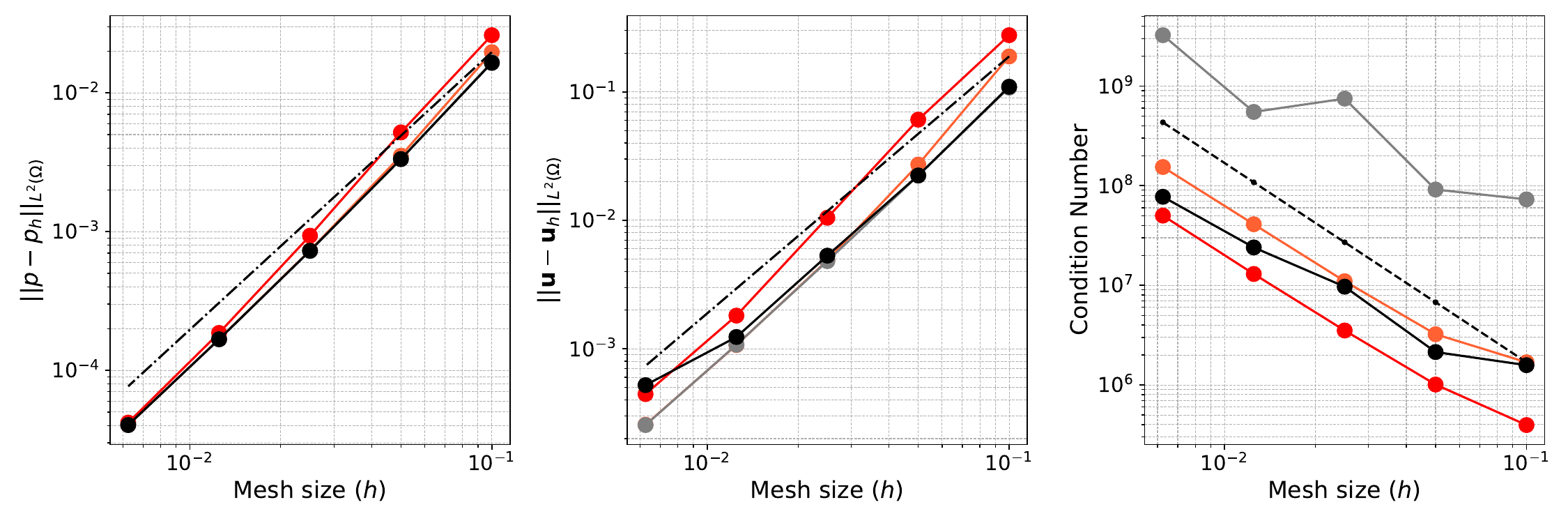} \\
	%\vspace{-0.5cm}
	\includegraphics[scale=0.6]{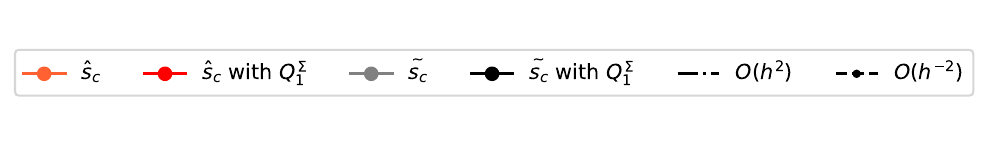}
	\caption{ 
		Example 1.1: Convergence and condition numbers versus mesh size $h$ using element triple $\RT_1\times Q_1\times Q^{\Sigma}_1$ and $\RT_1\times Q_1\times Q^{\Sigma}_2$. The alternative stabilization terms $\hat s_c$ and $\tilde s_c$ are compared.  
		Left: The $L^2$-error of the pressure. Middle:  The $L^2$-error of the velocity field. Right: The $1$-norm condition number.
		\label{fig:plot example2 3alts} }
\end{figure}

The choice of the normal vector in the bulk term for $\tilde s_c$ and $\hat s_c$ is not unique but is of importance. When we instead of using the gradient of the level set function as the normal, used the per element piecewise constant boundary normal, the velocity error was larger.

\subsubsection{Polynomial order of the Lagrange multiplier space}\label{sec:alt stab}
For almost all the examples we considered, switching the Lagrange multiplier space from $\Qks$ to $Q_{h,k}^{\Sigma}$ performs just as well. However, we illustrate in this example that when the right hand side data $\pmb{f}$ has a very large magnitude, then the accuracy with which the boundary condition is satisfied discretely is important for how much the velocity and the pressure error mix. In the context of the Stokes equations, these considerations are related to pressure robustness, and are discussed in \cite{FraNilZa23}. 

Let us consider the same circle geometry but with the following different data:
\begin{align}
	\bfu=0, \quad p=C(y^3-y^2/2+y-7/12)
\end{align}
with a positive constant $C>0$. We consider $k=0$ and investigate the velocity error for the element triples $\RT_0\times Q_0\times Q_0^{\Sigma}$ and $\RT_0\times Q_0\times Q_1^{\Sigma}$ for two values of $C=10^2,10^4.$

The velocity errors $\|\bfu\|_{L^2(\Omega)}$ are shown in Table \ref{tab:q0vsq1}. We see that the error is significantly reduced when using the higher order Lagrange multiplier space. This relation continues to hold with increasing $C$. 
We note in addition that in this example when the error in the velocity is only due to the perturbation of the boundary condition, the velocity convergence is of order two for the higher order Lagrange multiplier space, while it is of order one for the lower order Lagrange multiplier space.
\begin{table}[ht]
    \centering
    \begin{minipage}{0.45\textwidth}
        \centering
        \begin{tabular}{c|cc|cc}
            \multicolumn{1}{c|}{$C=10^2$} & \multicolumn{2}{c|}{Error} & \multicolumn{2}{c}{Convergence} \\\hline
            $h$ & $Q_1^{\Sigma}$ & $Q_0^{\Sigma}$ & $Q_1^{\Sigma}$ & $Q_0^{\Sigma}$ \\\hline
            0.1     & 1.73071   & 12.8871  & -     & -      \\
            0.05    & 0.36365   & 7.35506  & 2.25  & 0.81   \\
            0.025   & 0.0689639 & 3.7977   & 2.40  & 0.95   \\
            0.0125  & 0.0172261 & 1.96159  & 2.00  & 0.95   
        \end{tabular}
    \end{minipage}
    \hfill
    \begin{minipage}{0.45\textwidth}
        \centering
        \begin{tabular}{c|cc|cc}
            \multicolumn{1}{c|}{$C=10^4$} & \multicolumn{2}{c|}{Error} & \multicolumn{2}{c}{Convergence} \\\hline
            $h$ & $Q_1^{\Sigma}$ & $Q_0^{\Sigma}$ & $Q_1^{\Sigma}$ & $Q_0^{\Sigma}$ \\\hline
            0.1     & 173.071 & 1288.71  & -     & -      \\
            0.05    & 36.365  & 735.506  & 2.25  & 0.81   \\
            0.025   & 6.89639 & 379.77   & 2.40  & 0.96   \\
            0.0125  & 1.72261 & 196.159  & 2.00  & 0.95   
        \end{tabular}
    \end{minipage}
    \caption{Example 1.2: Comparison of the velocity error $\|\bfu\|_{L^2(\Omega)}$ and convergence order for the element triples $\RT_0\times Q_0\times Q_0^{\Sigma}$ and $\RT_0\times Q_0\times Q_1^{\Sigma}$ under different constants $C$.}
    \label{tab:q0vsq1}
\end{table}

\subsection{Example 2 (fitted versus unfitted)}
Consider the rectangle $\Omega=[0,1]\times[0,0.5]$. 
Let $\Omega_0=\Omdh = [0,1]\times[0,0.5+\epsilon]$ and $\Sigma = \{(x,y)\in\Omdh: y=0.5\}$ be the unfitted part of the boundary $\partial\Omega$. We compare the proposed unfitted method with the standard finite element method. For the standard FEM we use a fitted mesh on $\Omega$ and the boundary condition is imposed strongly on all boundaries (no Lagrange multipliers are used.) 
For the proposed cut finite element method we use a mesh that conforms to the left, right, and bottom boundary of the rectangle but not to the top boundary ( $\Sigma\subset\partial\Omega$ is unfitted). The boundary condition is enforced weakly on $\Sigma$ and strongly on $\partial\Omega\setminus\Sigma$. We let $\eta=1$ and take a linear divergence $g=2x+2y-1.5\in Q_{h,1}$ with the exact solution to problem \eqref{eqs:strongDarcy} being
\begin{align*}
	p=-(x^3/3-x^2/2 + y^3/3-y^2/4),\quad\quad \bfu = (x(x-1),y(y-0.5)).
\end{align*}
See Figure \ref{fig:ex1_sol} for a heat map of the velocity computed with the unfitted method.

\begin{figure}[ht]
	\centering 
	\includegraphics[scale=0.36]{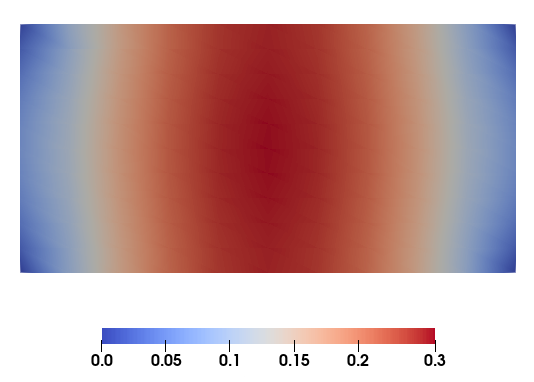}
	\caption{ 
		Example 2: Velocity solution magnitude using element triple $\RT_1\times Q_1\times Q_2^{\Sigma}$.
		\label{fig:ex1_sol} }            
\end{figure}

In Figure \ref{fig:ex1_upc} we observe for both element triples $\RT_0\times Q_0\times Q_1^{\Sigma}$ and $\RT_1\times Q_1 \times Q_2^{\Sigma}$ optimal convergence order for both pressure and velocity, as well as optimal condition number scaling. Aside from slightly larger condition numbers (see Figure~\ref{fig:ex1_upc}) and max-norm divergence errors (see Figure~\ref{fig:ex1_div}), the unfitted method performs as well as the fitted method. We emphasize that we have not optimized the stabilization constants for the unfitted discretization. We also saw in~\cite{FraHaNilZa22} that using a macroelement stabilization can reduce the max-norm of the divergence error. 

\begin{figure}[ht]
	\centering 
	\includegraphics[scale=0.40]{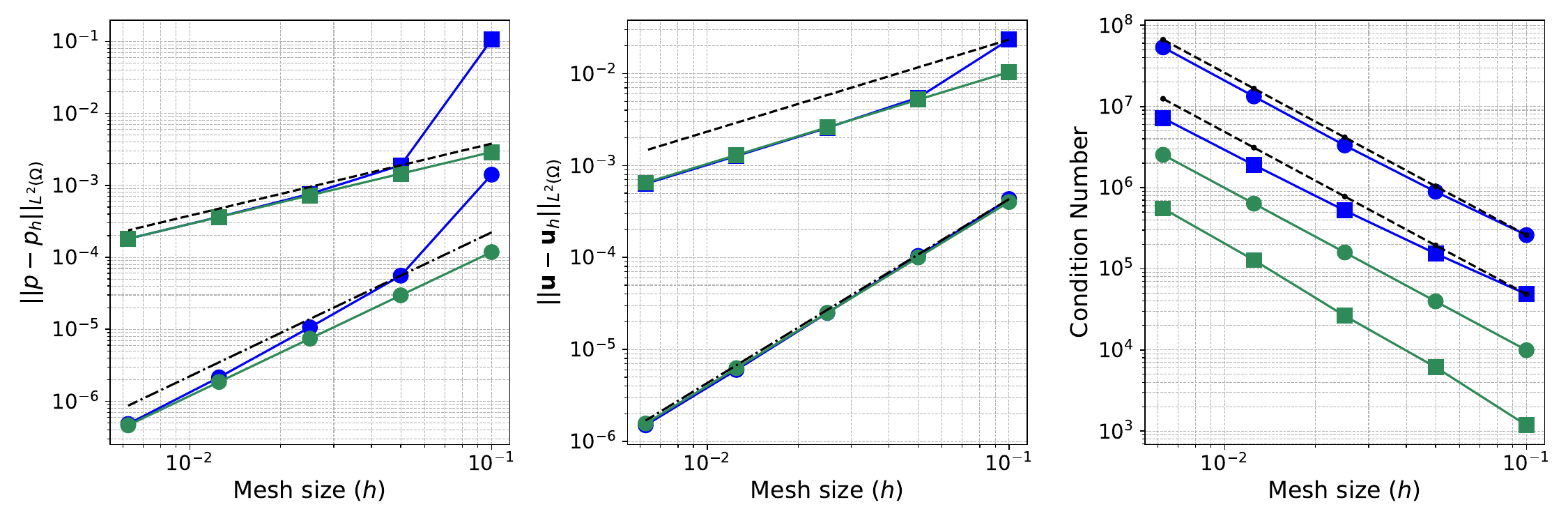} \\
	%\vspace{-0.5cm}
	\includegraphics[scale=0.6]{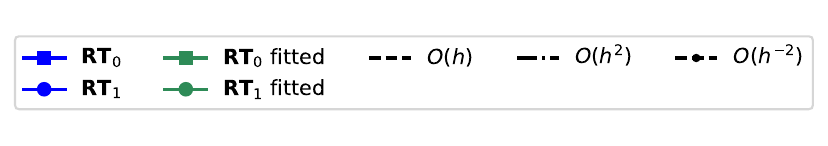}
	\caption{ 
		Example 2: Results using element triples $\RT_0\times Q_0\times Q_1^{\Sigma}$ and $\RT_1\times Q_1 \times Q_2^{\Sigma}$ versus fitted element pairs $\RT_0\times Q_0$ and $\RT_1\times Q_1$. Left: $L^2$-error of pressure. Middle: $L^2$-error of velocity. Right: $1$-norm condition number. 
		\label{fig:ex1_upc} }       
\end{figure}

We plot the $L^2-$ and the pointwise error of the divergence in Figure \ref{fig:ex1_div}. Since $\dive E\bfu\in Q_1$ only the triple $\RT_1\times Q_1\times Q_2^{\Sigma}$ achieves errors of order machine-epsilon. The rounding errors from the direct solver slightly increase with each mesh refinement. 
\begin{figure}[ht]
	\centering
	\includegraphics[scale=0.40]{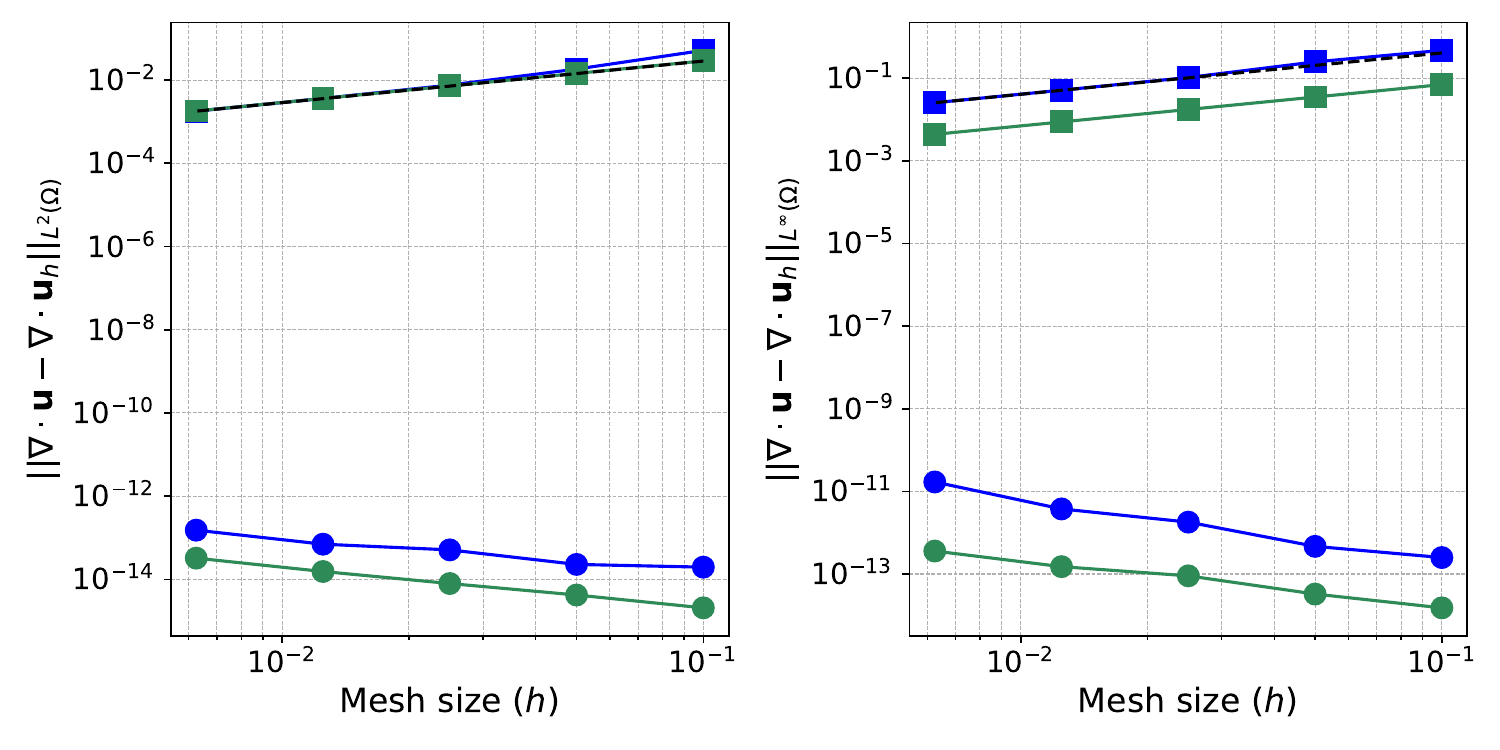} \\
	%\vspace{-0.5cm}
	\includegraphics[scale=0.6]{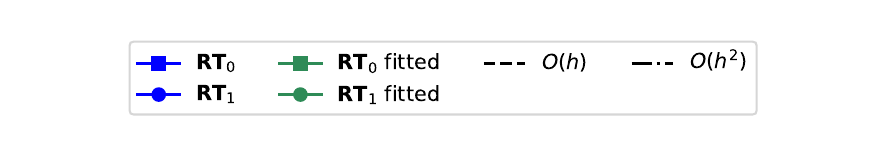}
	\caption{Example 2: The L$^2$-error and pointwise error of the divergence versus mesh size $h$, using element triples $\RT_0\times Q_0\times Q_1^{\Sigma}$ and $\RT_1\times Q_1 \times Q_2^{\Sigma}$. }
	\label{fig:ex1_div}      
\end{figure}

\subsection{Example 3 (interface problem with mixed BC)}
We study this time an annulus $\Omega$ with center $(0.5,0.5)$, inner radius $r_1=0.15$ and outer radius $r_2=0.45$, immersed in $\Omega_0=[0,1]^2$. We introduce a circular interface centered at the middle of the annulus with radius $r_{\Gamma}=0.3.$ We define $\Omega_1\subset\Omega$ as the part with points between $r_1$ and $r_\Gamma$, similarly $\Omega_2$ consist of the points between $r_\Gamma$ and $r_2$.  
Letting $\{p\}=(p_1+p_2)/2$ for $p_i=p|_{\Omega_i}, i=1,2$, the (of Neumann-type) interface conditions are as in \cite{DanSco12,FraHaNilZa22}: 
\begin{align}
	\jump{p} &= \eta_\Gamma\{\bfu\cdot\bfn\} \\
	\{p\} &= \hat{p}+\xi\eta_\Gamma \jump{\bfu\cdot\bfn}.
\end{align}
The problem parameters are $\eta=1$, and 
$$\eta_{\Gamma}= \frac{2r_{\Gamma}}{4\cos(r_{\Gamma}^2)+3},\ \hat{p} = \frac{19r_{\Gamma}^2+12\sin(r_{\Gamma}^2)+8\sin(2r_{\Gamma}^2)+24r_{\Gamma}^2\cos(r_{\Gamma}^2)}{4r_{\Gamma}^2(4\cos(r_{\Gamma}^2)+3)}, \ \xi = 1/8.$$

The exact solution is 
\begin{align*}
p=\sin(x^2+y^2) + \begin{cases} r^2/(2r^2_{\Gamma})+3/2 \ &\text{in }\Omega_1\\ r^2/r^2_{\Gamma}\ &\text{in }\Omega_2 \end{cases},\ \ \bfu = -(x,y)/r^2_{\Gamma}\begin{cases} 1+2\cos(x^2+y^2)\ &\text{in }\Omega_1\\ 2(1+\cos(x^2+y^2))\ &\text{in }\Omega_2 \end{cases},
\end{align*}
with $\dive\bfu=g$, where
\begin{equation}
  g=2/r^2_{\Gamma}\begin{cases}
	2(x^2+y^2)\sin(x^2+y^2)-2\cos(x^2+y^2)-1\ &\text{in }\Omega_1\\
	2((x^2+y^2)\sin(x^2+y^2)-\cos(x^2+y^2)-1)\ &\text{in }\Omega_2
  \end{cases}.
  \end{equation}
Note that $\dive\bfu=g\notin Q_k$, for any $k\geq 0$, we investigate the divergence condition in terms of its order of convergence.
\begin{figure}[ht]
	\centering 
	\begin{subfigure}[b]{0.45 \textwidth}  	
		\centering	
		\includegraphics[scale=0.28]{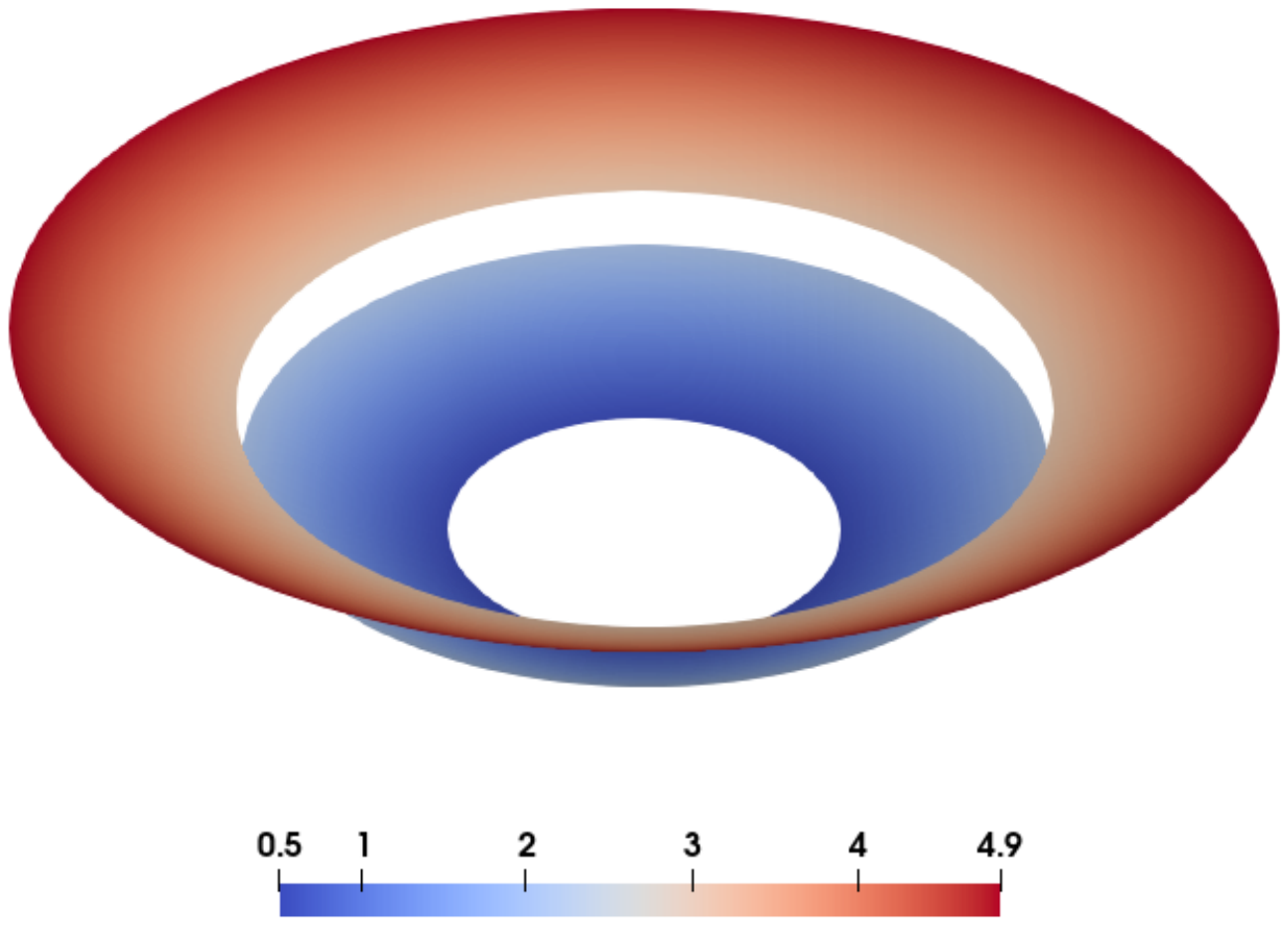}
		\caption{ 
			Pressure height map.
			 }      
	\end{subfigure}
	\begin{subfigure}[b]{0.45 \textwidth}  	
		\centering	
		\includegraphics[scale=0.28]{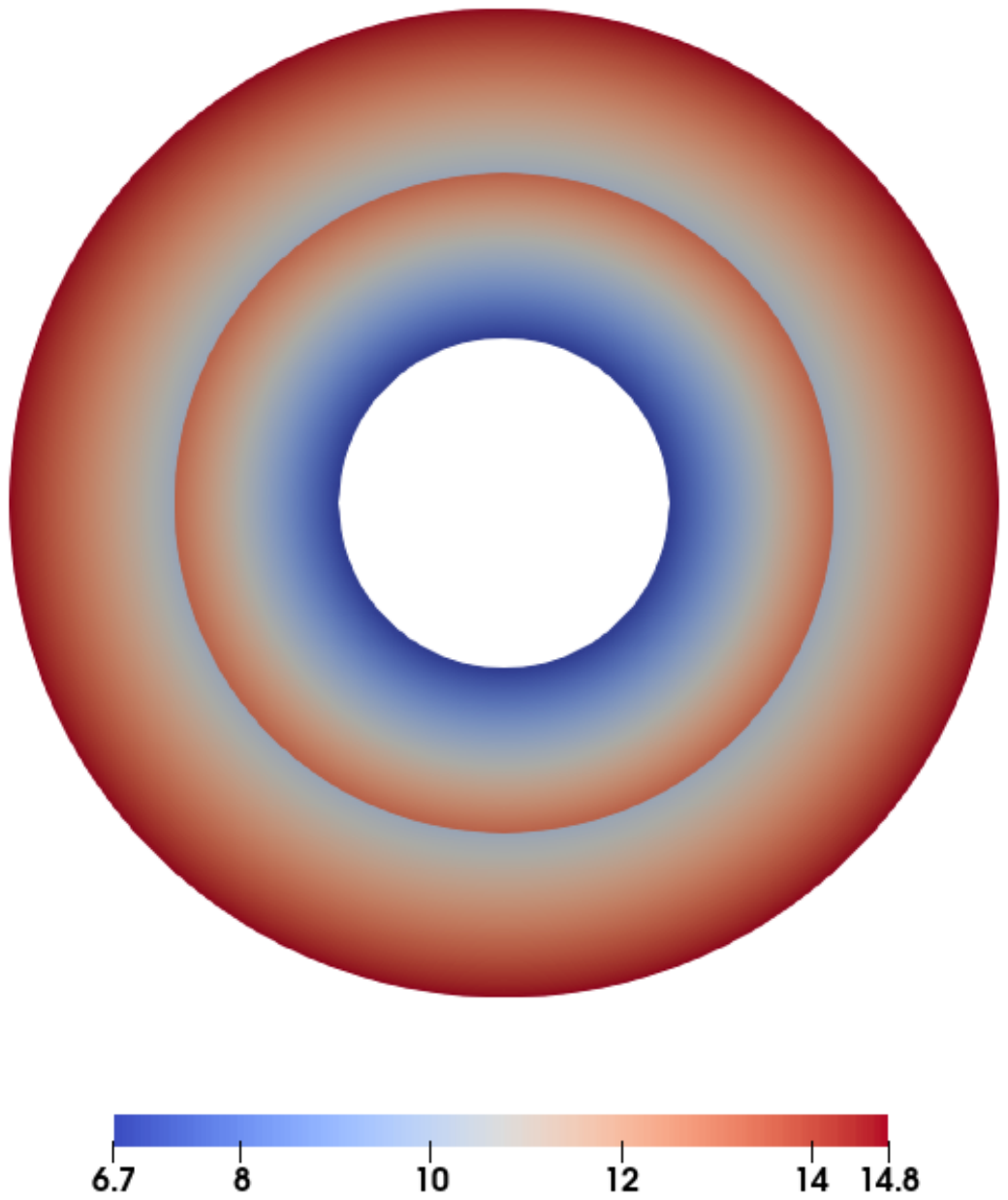}
		\caption{ 
			Velocity magnitude heat map.
			 }    
	\end{subfigure}     
	\caption{ 
		Example 3: Solution using element triple $\RT_1\times Q_1\times Q_2^{\Sigma}$.
	\label{fig:ex3_sol} }            
\end{figure}
On the inner circle with radius $r_1$ we set the Neumann condition $p_h=p$ and on the outer circle with radius $r_2$ we set the Dirichlet condition $\bfu_h\cdot\bfn = \bfu\cdot\bfn=u_B.$ As such, we choose the approximation space for the pressure to be $\Qk=\bigoplus_{T \in \mcT_{h}} Q_k(T)$ since there is no uniqueness issue and the pressure average is not set due to the presence of the Neumann condition. See Figure \ref{fig:ex3_sol} for plots of the solution obtained with the element triple $\RT_1\times Q_1\times Q_2^{\Sigma}$. We compare the proposed Lagrange multiplier method \eqref{eq:discretedarcy2} with the penalty method~\eqref{eq:discretedarcy_nitsche}. The penalty parameter is set to be $100$.

Errors are shown in Figures \ref{fig:plot example3 convergence} and \ref{fig:diverror example 3}. We see mostly comparable results except for the velocity and condition number when using $\RT_0$. Here the penalty method performs noticeably worse and seems to be unstable. This behavior of the penalty method is not present when using $\RT_1$. The divergence errors are however completely equivalent for both methods, as can be seen in Figure \ref{fig:diverror example 3}.

\begin{figure}[ht]
	\centering 
	\includegraphics[scale=0.40]{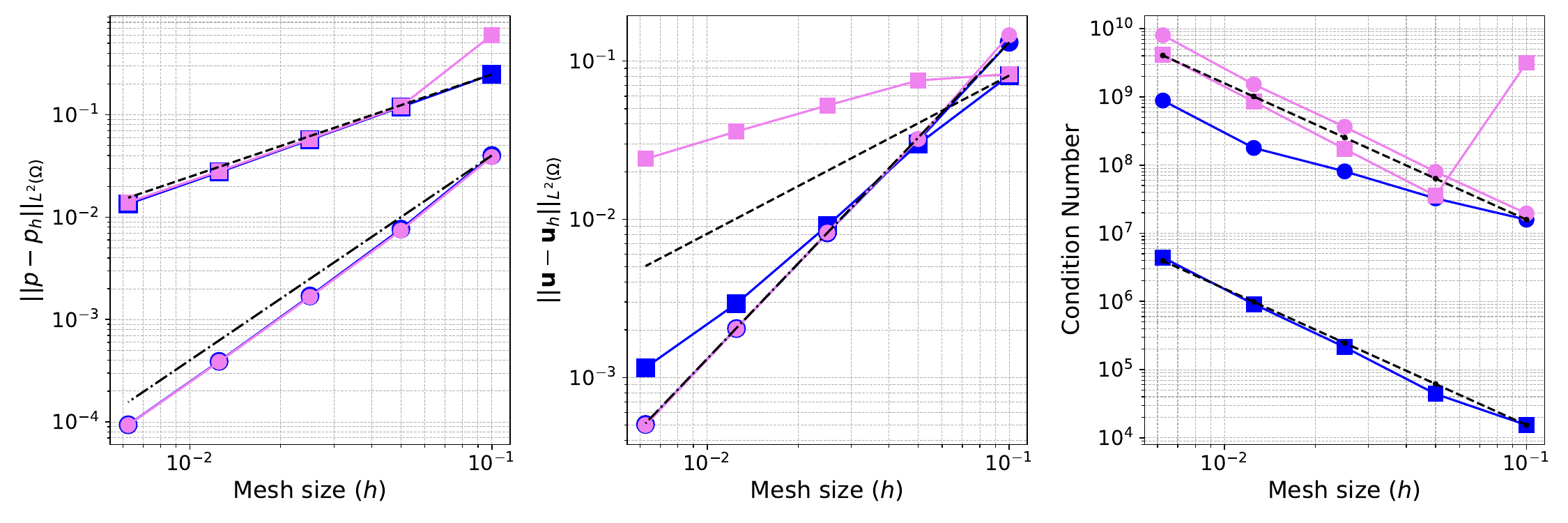} \\
	%\vspace{-0.5cm}
	\includegraphics[scale=0.6]{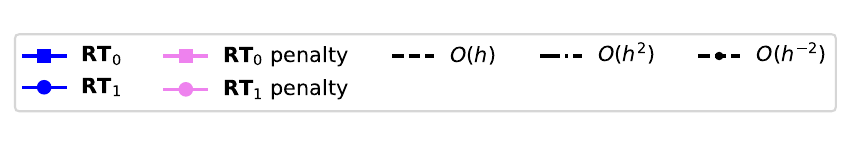}
	\caption{ 
		Example 3: Convergence and condition numbers using element triples $\RT_0\times Q_0 \times Q^{\Sigma}_1$ and $\RT_1\times Q_1\times Q^{\Sigma}_2$ versus mesh size $h$. 
		Left: The $L^2$-error of the pressure. Middle:  The $L^2$-error of the velocity field. Right: The $1$-norm condition number.
		\label{fig:plot example3 convergence} }                     
\end{figure}

\begin{figure}[ht]
	\centering
	\includegraphics[scale=0.40]{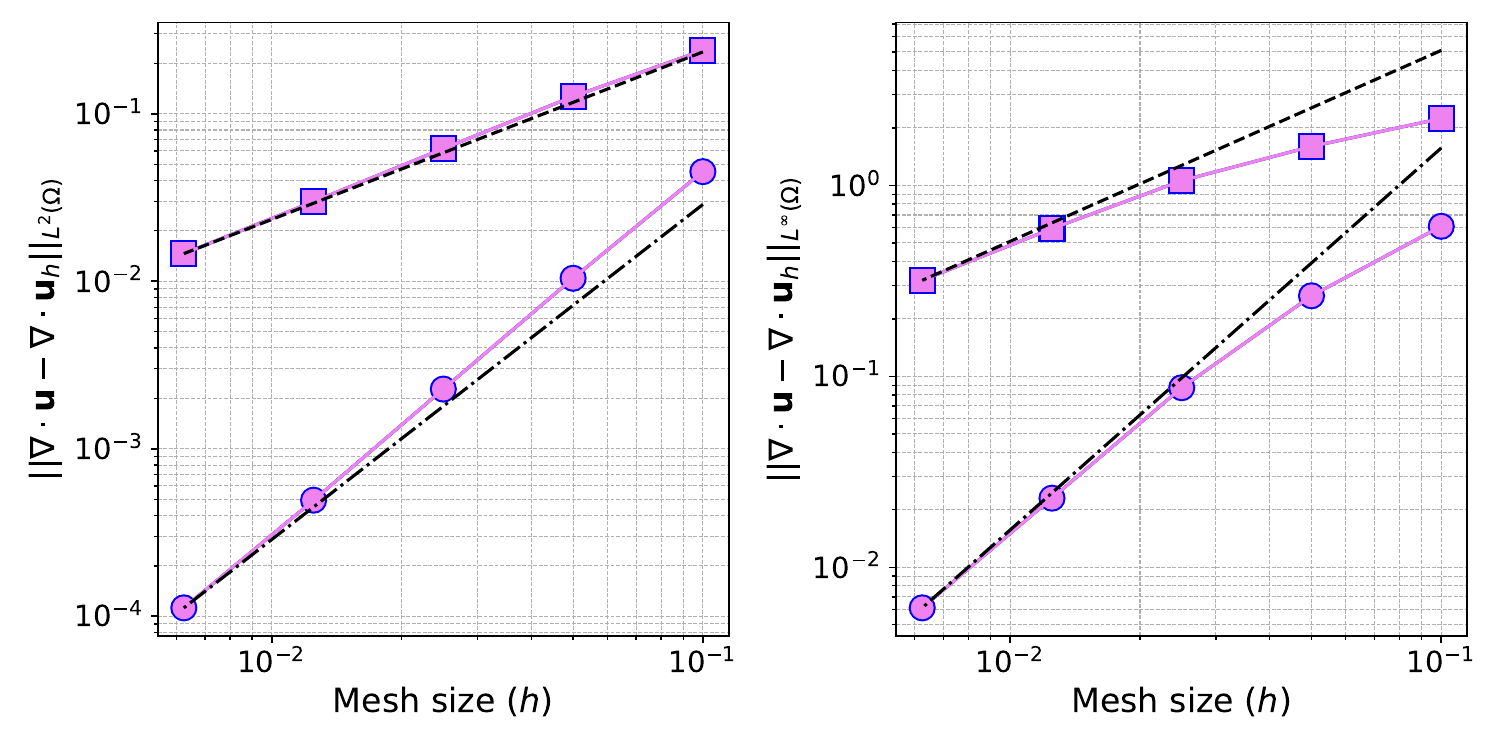}\\
	\includegraphics[scale=0.6]{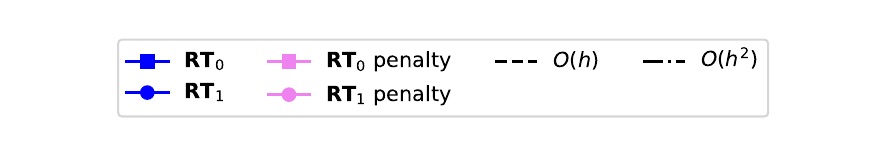}
	\caption{Example 3: The divergence error versus mesh size $h$ using element triples $\RT_0\times Q_0 \times Q^{\Sigma}_1$ and $\RT_1\times Q_1\times Q^{\Sigma}_2$. 
		Left: The $L^2$-error of the divergence. Right: The pointwise error of the divergence.}
	\label{fig:diverror example 3}      
\end{figure}

\section{Conclusion}\label{sec:conclusion}
We have introduced a divergence preserving cut finite element method for the Darcy fictitious domain problem based on the Raviart-Thomas elements $\RT_k$ for the velocity and the piecewise polynomial space $\Qk$ for the pressure, with $k\geq0$. The method uses a stabilized Lagrange multiplier to impose Dirichlet boundary conditions, which results in a perturbed composite saddle-point problem. In order to capture the boundary condition sufficiently well we use a higher polynomial degree than $k$ in the Lagrange multiplier space. We choose the space $\Qks$. We have shown, for the lowest order case $k=0$, that the method is able to retain optimal approximation properties for velocity and pressure, preserve the divergence-free property of the underlying elements, and result in well-posed linear systems where the condition number of the resulting system matrix is bounded independently of the mesh size. 
For the Lagrange multiplier, we studied several stabilization terms that can be used also when higher order elements than piecewise constants are used for the Lagrange multiplier space. 

The described properties of the method are showcased with three numerical examples. We have compared the proposed method to a standard finite element method, an alternative stabilization method for the Lagrange multiplier, and a penalty method for imposing Dirichlet boundary conditions. The results show that the proposed method is able to produce accurate solutions with optimal convergence rates, in a comparable or favorable way to the other methods. 

\begin{appendix}
\section{Auxiliary lemmas} \label{sec:auxlemmas}
We state here some lemmas we need to prove the inf-sup condition Lemma \ref{lem:B_infsup}.

We gather \cite[Lemmas 3.16, 3.17]{Gatica14} in the following collection of local interpolation errors. We take the convention that $H^0(\Omdh)=L^2(\Omdh).$
\begin{lemma}\textbf{(Local interpolation error estimates)}\label{lem:loc_interp}
	Fix $T\in\mesh$. Let $\bfu\in\bfH^{k+1}(T)$ with $\dive\bfu\in H^{k+1}(T)$, then for $0\leq m\leq \ell+1$ and $0\leq \ell \leq k$
	\begin{align*}
		| \bfu-\interpv \bfu |_{m,T} 
		&\lesssim h^{\ell+1-m}|\bfu|_{\ell+1,T} \\
		| \dive\bfu-\dive\interpv \bfu |_{m,T},
		&\lesssim h^{\ell+1-m}|\dive\bfu|_{\ell+1,T}.
	\end{align*} 
\end{lemma}

A direct consequence is that the global interpolation operator associated to $\Vk$ is stable in the following sense, 
\begin{equation} \label{eq:stabilityintp}
	\|\interpv\bfv\|_{\bfH^{\dive}(\Omdh)} \lesssim \|\bfv\|_{\bfH^1(\Omdh)}, \quad  \forall \bfv \in \bfH^1(\Omdh).
\end{equation}
\begin{proof}
	For a proof, see \cite[Lemma 4.4]{Gatica14}.
\end{proof}

\begin{lemma}(\textbf{The Poincaré inequality}, \cite[Theorem 7.91]{Salsa2016})\label{lem:Poincare}
	Let $ S\in\RR^d, d=2,3$, be a bounded Lipschitz domain with $\emptyset\neq B\subset \partial S$ and let $\bfv\in\bfH_{0,B}^1( S)$ so that $\bfv$ has zero trace on $B$. Then
	\begin{align*}
		\|\bfv\|_{ S} \leq C_P \|\nabla\bfv\|_{ S},\quad \text{ where } C_P=2\text{diam}( S).
	\end{align*} 
\end{lemma}

\begin{lemma}(\textbf{Stable lifting}, \cite[Lemma 4.7]{DanSco12})\label{lem:divergence_Poincare}
  Let $ S\in\RR^d, d=2,3,$ be a bounded convex domain or a smooth domain. 
  Then for any $q\in L^2( S)$ there exists $\bfv \in \bfH_{0,B}^1( S)$ with $\emptyset\neq B\subset\partial S$, satisfying 
	\[\dive\bfv = q,\ \|\bfv\|_{\bfH^1( S)}\leq C\|q\|_{ S} .\] 
	where $C$ is a positive constant independent of $q$. If $q$ satisfies $\int_D q=0$ then $\bfv \in \bfH_{0}^1(D)$.
\end{lemma}

\section{Alternative stabilization techniques} \label{sec:alt-stab-tech}
\subsection{Patch stabilization \cite{preuss2018}} \label{sec:patch-stab}
For the bulk variables the face-based stabilization is equivalent to a polynomial extension based stabilization. Any polynomial has a canonical extension; consider for instance $\bfv_h\in \RT_0(\Omdh)$ and its restriction $\bfv_h^T=\bfv_h|_T\in\RT_0(T)$ for $T\in\mesh$, then 
\begin{align}
	E\bfv_h^T(x) = \bfv_h^T(x),\quad \forall x\in \Omdh\setminus T,
\end{align}
can be defined simply by evaluating the polynomial $\bfv_h^T$ at the point $x$. 
For the first two terms \eqref{eq:full stab-u}-\eqref{eq:full stab-b} we define the extension based stabilization as
\begin{align}
	\tilde{s}(\bfu_h,\bfv_h) &=  \sum_{F \in \mcFs}
	\tau  ([\![\bfu_{h} ]\!], [\![\bfv_{h}]\!])_{P(F)}, \label{eq:patch_1} \\
	\tilde{s}_b(\bfu_h,q_h) &=  \sum_{F \in \mcFs}
	\tau_b  ( [\![\dive \bfu_{h}]\!], [\![q_{h}]\!] )_{P(F)},\label{eq:patch_2}
\end{align} 
where $P(F)=T_1\cup T_2$ is the union (or patch) of the two elements sharing the face $F$ and 
\begin{align}
	[\![q_h]\!]|_{T_i} = q_h^{T_i}-Eq_h^{T_j},\quad\quad i,j=1,2, i\neq j.
\end{align}
The extension based stabilization operators are beneficial to use numerically since the number of terms are the same regardless of the order of the elements and no higher order derivatives have to be implemented. In our numerical experiments we use these extension based stabilization terms.

\subsection{Macroelement partitioning}\label{sec:macro-elements}
Each element in the mesh $\mcT_{h}$ can be classified either as having a large intersection with the domain $\Omega$, or a small intersection. We say that an element $T \in \mcT_{h}$ has a large $\Omega$-intersection if
\begin{equation}\label{eq:largeel}
	\delta  \leq \frac{|T \cap \Omega|}{|T|},
\end{equation}
where $\delta$ is a positive constant which is independent of the element and the mesh parameter. We collect all elements with a large intersection in 
\begin{equation}
	\mcTL=\left\{T \in \mcT_{h} : |T \cap \Omega| \geq \delta |T| \right\}.
\end{equation}
Using such a classification we create a macroelement partition $\mcM_{h}$ of $\Omega_{\mcT_{h}}$
\begin{itemize}
	\item To each $T_L\in\mcTL$ we associate a macroelement mesh $\mcT_{h}(T_L)$ containing $T_L$ and possibly adjacent elements that are in $\mcT_{h}\setminus \mcTL$, i.e., elements classified as having a small intersection with $\Omega$ and  are connected to $T_L$ via a bounded number of internal faces. 
	\item Each element $T \in \mcT_{h}$ belongs to precisely one macroelement mesh $\mcT_{h}(T_L)$.
	\item Each macroelement $M_L \in \mcM_{h}$ is the union of elements in $\mcT_{h}(T_L)$, i.e., 
	\begin{equation}
		M_L = \bigcup_{T \in \mcT_{h}(T_L)} T.
	\end{equation}
\end{itemize}
We denote by $\mcFM$ the set consisting of interior faces of $M_L\in\mcM_{h}$. Note that $\mcFM$ is empty when $T_L$ is the only element in $\mcTT$.
%See Figure~\ref{fig:macro_mesh} for an illustration of a macro-element partitioning. Faces in $\mcFM$ are colored and black thick lines illustrate the boundary of macro-elements consisting of more than one element. 

We follow Algorithm 1 in \cite{LarZah23} when we construct the macroelement partition, but here follows a short heuristic explanation. 

The algorithm to construct a macroelement partition $\mcM_h$ starts by marking all cut elements which are small and saving these in a list of small elements each to be connected to some large element $T_L \in \mcTL$. Then the algorithm takes one small element and checks its face neighbors until a neighbor with a large $\Omega$-intersection is found \textit{and stored}, after which it removes the small element from the list. Then while there are still small elements in the list, it repeats this procedure. If the algorithm doesn't find a face neighbor with a large $\Omega$-intersection to some small element $K$, it skips the element for the moment and returns after other small element has had a chance to find direct face neighbors with large $\Omega$-intersection. Then previously marked small elements will no longer be in the list, and so $K$ will eventually find big neighbors and be connected via faces of these to the same large element $T_L$.

\begin{remark}[Macroelement stabilization] \label{rmk:macrostab}
	As introduced in \cite{LarZah23},  one can utilize this macroelement partition of the mesh and stabilize on fewer elements. Stabilization is then applied only on internal faces of macroelements and never on faces shared by neighboring macroelements. The stabilization corresponding to \eqref{eq:full stab-u} and \eqref{eq:full stab-b} becomes
	\begin{align}
		s_{\pmb{u}}(\pmb{u}_h,\pmb{v}_h) &=  \sum_{F \in \mcFM} \sum_{j=0}^{\kp+1} \tau_{u} h^{2j + 1} (\jump{D^j_{\bfn_F} \pmb{u}_h }, \jump{D^j_{\bfn_F} \pmb{v}_h})_{F}, \\
		s_{b}(\pmb{u}_h,q_h) &=   \sum_{F \in \mcFM}  \sum_{j=0}^{\kp} \tau_{b} h^{2j + \gamma} ( \jump{D^j (\dive \pmb{u}_h)}, \jump{D^j q_h} )_{F}, \label{eq:stab-b} \\
		s_{p}(p_h,q_h) &=   \sum_{F \in \mcFM} \sum_{j=0}^{\kp} \tau_{b} h^{2j + \gamma} (\jump{D^j p_h}, \jump{D^j q_h})_{F}.
	\end{align}
	Notably, no proof in any prior section is affected since a version of Lemma \ref{lem:sp_ineq} can be shown to hold for this present relaxed stabilization.
\end{remark}

\begin{remark}[Local mass preservation]
	We will always have a local mass preservation in each macroelement $M\in\mcM_h$, which can be seen from taking $q_h = 1$ on $M$ and $0$ everywhere else in $\mesh$. Then, as for standard discontinuous elements (c.f. \cite{JoLiMeNeRe17}), we get from \eqref{eq:discretedarcy2} that  
	\begin{equation}
		\int_M \dive\bfu_h = \int_M g_E.
	\end{equation}
\end{remark}

\end{appendix}
 
\section*{Acknowledgements}
 This research was supported by the Swedish Research Council Grant No. 2022-04808 and the Wallenberg Academy Fellowship KAW 2019.0190.

\addcontentsline{toc}{section}{Bibliography}

\bibliographystyle{abbrv}
\bibliography{references}

\begin{thebibliography}{10}

\bibitem{BofBreFor13}
D.~Boffi, F.~Brezzi, and M.~Fortin.
\newblock {\em Mixed finite element methods and applications}.
\newblock Springer series in computational mathematics, 44. Springer, New York,
  2013.

\bibitem{BreSco}
S.~C. Brenner and L.~R. Scott.
\newblock {\em The Mathematical Theory of Finite Element Methods}.
\newblock Springer-Verlag, 2008.

\bibitem{Bu10}
E.~Burman.
\newblock Ghost penalty.
\newblock {\em C. R. Acad. Sci. Paris, Ser. I}, 348(21-22):1217 -- 1220, 2010.

\bibitem{BuHa12}
E.~Burman and P.~Hansbo.
\newblock Fictitious domain finite element methods using cut elements: {II}.
  {A} stabilized {N}itsche method.
\newblock {\em Applied Numerical Mathematics}, 62(4):328--341, 2012.

\bibitem{BurHanLarst24}
E.~Burman, P.~Hansbo, and M.~Larson.
\newblock Cut finite element method for divergence-free approximation of
  incompressible flow: A {L}agrange multiplier approach.
\newblock {\em SIAM Journal on Numerical Analysis}, 62(2):893--918, 2024.

\bibitem{BurHanLarMas2018MFD}
E.~Burman, P.~Hansbo, M.~G. Larson, and A.~Massing.
\newblock Cut finite element methods for partial differential equations on
  embedded manifolds of arbitrary codimensions.
\newblock {\em ESAIM: Mathematical Modelling and Numerical Analysis},
  52(6):2247--2282, 2018.

\bibitem{DanSco12}
C.~D'Angelo and A.~Scotti.
\newblock A mixed finite element method for darcy flow in fractured porous
  media with non-matching grids.
\newblock {\em ESAIM: Mathematical Modelling and Numerical Analysis},
  46(2):465--489, 2012.

\bibitem{Ern2006Evaluation}
A.~Ern and J.-L. Guermond.
\newblock Evaluation of the condition number in linear systems arising in
  finite element approximations.
\newblock {\em ESAIM: Mathematical Modelling and Numerical Analysis},
  40(1):29--48, 2006.

\bibitem{evans2022pde}
L.~C. Evans.
\newblock {\em Partial differential equations}, volume~19.
\newblock American Mathematical Society, 2022.

\bibitem{repo}
T.~Frachon.
\newblock Cut{FEM}-{L}ibrary.
\newblock {https://github.com/CutFEM/CutFEM-Library}, 2024.

\bibitem{FraHaNilZa22}
T.~Frachon, P.~Hansbo, E.~Nilsson, and S.~Zahedi.
\newblock A divergence preserving cut finite element method for {Darcy} flow.
\newblock {\em SIAM J. Sci. Comput.}, 46(3):A1793--A1820, 2024.

\bibitem{FraNilZa23}
T.~Frachon, S.~Zahedi, and E.~Nilsson.
\newblock Divergence-free cut finite element methods for {Stokes} flow.
\newblock {\em arXiv preprint arXiv:2304.14230}, 2023.

\bibitem{Gatica14}
G.~N. Gatica.
\newblock A simple introduction to the mixed finite element method.
\newblock {\em Theory and Applications. Springer Briefs in Mathematics.
  Springer, London}, 2014.

\bibitem{GrandeLehr2018Surf}
J.~Grande, C.~Lehrenfeld, and A.~Reusken.
\newblock Analysis of a high-order trace finite element method for {PDE}s on
  level set surfaces.
\newblock {\em SIAM Journal on Numerical Analysis}, 56(1):228--255, 2018.

\bibitem{LiuNeiOls23}
{H. Liu}, {M. Neilan}, and {M. Olshanskii}.
\newblock A {CutFEM} divergence-free discretization for the {Stokes} problem.
\newblock {\em ESAIM: Mathematical Modelling and Numerical Analysis},
  57(1):143--165, 2023.

\bibitem{HanHanLar2003TraceIneq}
A.~Hansbo, P.~Hansbo, and M.~G. Larson.
\newblock A finite element method on composite grids based on {Nitsche's}
  method.
\newblock {\em ESAIM: Mathematical Modelling and Numerical Analysis},
  37(3):495--514, 2003.

\bibitem{HaLaZa14}
P.~Hansbo, M.~G. Larson, and S.~Zahedi.
\newblock A cut finite element method for a {S}tokes interface problem.
\newblock {\em Applied Numerical Mathematics}, 85:90--114, 2014.

\bibitem{Hiptmair2012Universal}
R.~Hiptmair, J.~Li, and J.~Zou.
\newblock Universal extension for {Sobolev} spaces of differential forms and
  applications.
\newblock {\em Journal of Functional Analysis}, 263(2):364--382, 2012.

\bibitem{JoLiMeNeRe17}
V.~John, A.~Linke, C.~Merdon, M.~Neilan, and L.~G. Rebholz.
\newblock On the {divergence} {constraint} in {mixed} {finite} {element}
  {methods} for {incompressible} {flows}.
\newblock {\em SIAM Review}, 59(3):492--544, Jan. 2017.

\bibitem{LaZa2020SurfStab}
M.~G. Larson and S.~Zahedi.
\newblock Stabilization of high order cut finite element methods on surfaces.
\newblock {\em IMA Journal of Numerical Analysis}, 40(3):1702--1745, 2020.

\bibitem{LarZah23}
M.~G. Larson and S.~Zahedi.
\newblock {Conservative cut finite element methods using macroelements}.
\newblock {\em Computer Methods in Applied Mechanics and Engineering},
  414:116141, 2023.

\bibitem{Lehrenfeld2023Divfree}
C.~Lehrenfeld, T.~van Beeck, and I.~Voulis.
\newblock Analysis of divergence-preserving unfitted finite element methods for
  the mixed {Poisson} problem.
\newblock {\em arXiv preprint arXiv:2306.12722}, 2023.

\bibitem{linke_pressure-robustness_2016}
A.~Linke and C.~Merdon.
\newblock Pressure-robustness and discrete {Helmholtz} projectors in mixed
  finite element methods for the incompressible {Navier}--{Stokes} equations.
\newblock {\em Computer Methods in Applied Mechanics and Engineering},
  311:304--326, Nov. 2016.

\bibitem{MaLaLoRo14}
A.~Massing, M.~Larson, A.~Logg, and M.~Rognes.
\newblock A stabilized {Nitsche} fictitious domain method for the {Stokes}
  problem.
\newblock {\em Journal of Scientific Computing}, 61(3):604--628, 2014.

\bibitem{preuss2018}
J.~Preu{\ss}.
\newblock Higher order unfitted isoparametric space-time {FEM} on moving
  domains.
\newblock {\em Master's thesis, University of Gottingen}, 2018.

\bibitem{Reu15}
A.~Reusken.
\newblock {Analysis of trace finite element methods for surface partial
  differential equations}.
\newblock {\em IMA Journal of Numerical Analysis}, 35(4):1568--1590, 2015.

\bibitem{Salsa2016}
S.~Salsa.
\newblock {\em Partial differential equations in action: from modelling to
  theory}, volume~99.
\newblock Springer, 2016.

\end{thebibliography}

\end{document}